\documentclass[12pt]{article}

\usepackage{hyperref}
\usepackage{optidef}
\usepackage{amsmath,amssymb,amsfonts,amsthm,enumerate}
\usepackage{eucal} 
\usepackage{xcolor}
\usepackage{tikz-cd}
\usepackage{youngtab}
\usepackage{young}
\usepackage{lscape}
\usepackage{tikz}
\usepackage{environ}
\usepackage{caption}
\usepackage{subcaption}
\usepackage{algorithm}
\usetikzlibrary{positioning, backgrounds}
\usepackage[noend]{algpseudocode}
\usepackage{graphicx} 
\usepackage{adjustbox}
\usepackage{amsmath,amsthm,amssymb,amscd}

\DeclareMathOperator{\chull}{conv}

\numberwithin{equation}{section}

\newtheorem{thm}{Theorem}
\newtheorem{lemma}{Lemma}

\newtheorem{cor}[lemma]{Corollary}

\theoremstyle{definition}
\newtheorem{defn}{Definition}

\theoremstyle{definition}

\DeclareMathOperator{\KL}{KL}

\newcommand{\pspace}{\mathcal{P}}
\newcommand{\qspace}{\mathcal{Q}}

\newcommand{\kl}[2]{\KL(#1 \lVert #2)}
\DeclareMathOperator{\Id}{Id}
\newcommand{\PP}{\mathcal{P}}
\newcommand{\QQ}{\mathcal{Q}}
\newcommand{\RR}{\mathcal{R}}
\newcommand{\psimp}{\Delta_n^+}

\title{Transport based embeddings with topological guarantees}

\author{Erik Carlsson and John Carlsson}
\begin{document}
\maketitle
\begin{abstract}
Point clouds arising in  image collections, samples
from Markov chain Monte Carlo, or states of a random walk, often have
a simple underlying geometry which is obscured by noise, high
ambient dimension, and the failure of Euclidean distance to reflect
similarity. Methods such as UMAP and t-SNE condense such data into
usable form, but rely on heuristic choices and provide no guarantee
that the output reflects the topology of the input. We introduce a
condensation method that comes with such a guarantee. Encoding the
data as a positive $m\times n$ stochastic matrix $Q=(q_{ij})$, for
instance the transition matrix of a random walk on the point cloud,
we define a potential function
$\psi(p)=\log \sum_{i} \exp(-\kl{p}{q_{i\bullet}})$
on the probability simplex $\Delta_n$, where $\kl{p}{q}$ is the
Kullback-Leibler divergence, and prove that $\psi$ is $c$-convex in
the sense of Optimal Transport Theory for the cost function
$c(p,q)=\kl{p}{q}$. The associated transport map collapses noisy
directions while provably preserving topology: the super-level sets
of $\psi$ are homotopy equivalent to those of the $c$-conjugate
function, whose image is a condensed, resampleable family of topological spaces which can be interpreted as a continuous analog of an alpha shape. 
We demonstrate the method by recovering
the circle of camera angles from the COIL image dataset, where a
standard PCA pipeline produces spurious homology, and the quotient
$SO(3)/A_4$ from $45{,}000$ views of a tetrahedron in the SYMSOL
pose-estimation benchmark.

\end{abstract}

\section{Introduction}

A recurring task in data analysis is to recover the intrinsic
geometry of a point cloud: a low-dimensional shape, possibly with
nontrivial topology, obscured by noise and a high-dimensional
embedding. Widely used methods such as UMAP \cite{McInnes2018UMAPUM}
and t-SNE \cite{hinton2008tsne} address this by condensing the data
into low-dimensional coordinates, but they rely on heuristic
choices and provide no guarantee that the output reflects the
topology of the input, while persistent homology computes
topological invariants with stability guarantees but does not by
itself produce a denoised representation of the data. The purpose
of this paper is a construction that does both: it converts a point
cloud into a filtered family of condensed ``shapes'' from which
arbitrarily many samples can be drawn, together with a theorem
guaranteeing that each shape is homotopy equivalent to a
super-level set of a density-like potential on the data. Because
the input is a stochastic matrix rather than a Euclidean point
cloud, the construction applies to data described by any similarity
kernel or random walk, such as image sets with symmetries,
where Euclidean distance fails to reflect the underlying geometry.

In previous work \cite{carlsson2025kernel}, the authors used Legendre duality to transform a positively weighted 
sum of (isotropic) Gaussian kernels 
\begin{equation}
\label{eq:gaussiansum}    
\phi(x)=\log f(x),\  f(x)=\sum_{i=1}^n a_i\exp(-t\lVert x-x_i\rVert^2)
\end{equation}
into a filtered family 
$\mathcal{Y}(a)\subset \chull(\{x_i\})\subset \mathbb{R}^d$
of condensed subspaces of the interior of the convex hull,
which were interpreted as a continuous version of an alpha shape \cite{edelsbrunner1983shape}. It was proved that each $\mathcal{Y}(a)$ is homotopy equivalent to the super-level set 
$\mathcal{X}(a)=\phi^{-1}[a,\infty)$, but that their covering number is stable with respect to the embedding dimension, making it plausible to triangulate for $d\gg 0$. By contrast, 
the $\epsilon$-covering number of 
$\mathcal{X}(a)$ generally scales exponentially with $d$ even when the point cloud consists of a single point.


Extending this construction
to other geometries or kernel functions requires replacing the
Legendre transform by a more general type of duality arising in Optimal Transport Theory known as $c$-convexity 
\cite{villani2008old}, 
in which the negative exponent 
$t\lVert x-y\rVert^2$ of the kernel is replaced by a more general 
\emph{cost function} $c(x,y)$.
In this setup, while the potential corresponding to
$\phi$ is not derived as the Kantorovich potential 
of any particular optimal transport problem, the machinery
still applies and produces a conjugate function whose super-level sets define the shape $\mathcal{Y}(a)$, as well as
a \emph{transport map}
$T:\mathcal{X}(a)\rightarrow \mathcal{Y}(a)$. Exactly 
as in the Gaussian case, the transport map
simultaneously plays the role of the homotopy equivalence between the two spaces, and can also be used to collapse away topologically trivial dimensions in $\mathcal{X}(a)$, which would be interpreted as noise in certain data-related applications.
In a recent example, 
G. DePaul and the first author formulated one
such extension using cosine similarity-based kernel and cost functions on the sphere, 
with an application to Neuro Science
\cite{carlsson2025cosim}.


Extending the transport method to point clouds which are not well-captured by Euclidean distances can be handled by
considering more general kernel functions $K_t:\mathbb{R}^d\times \mathbb{R}^d\rightarrow \mathbb{R}_{>0}$
in \eqref{eq:psigauss}.
Such a construction could then be applied to 
the super-level set homotopy type of
essentially arbitrary densities $\rho$ on $\mathbb{R}^d$,
for instance by using the Metropolis algorithm to express it as a convolution. 
Indeed, if $K_t(x,y)$ denotes the transition kernel at time $t$ of a 
Markovian process with stationary density $\rho$, then
\begin{equation}
\label{eq:movingkernel}
\rho(y)=\int_{\mathbb R^d}\rho(x)K_t(x,y)dx.
\end{equation}
Then given samples $(x_1,\ldots,x_N)$ from the underlying distribution 
of $\rho$, we find that
$\rho(x)\approx \frac1N\sum_{i=1}^N K_t(x_i,x)$,
which is precisely a kernel density estimator with moving kernels. It is not clear how to make such an 
extension, since the arguments of \cite{carlsson2025kernel,carlsson2025cosim} were highly specific to the kernel function.

We propose another way to extend the pipeline from Gaussian kernels to 
moving ones: let us identify
$\theta :\mathbb{R}^d\rightarrow\mathcal{X}$,
where $\mathcal{X}$ consists of probability distributions on 
$\mathbb{R}^d$ given by Gaussian kernels with fixed scale $t$, by letting $\theta(x)$ denote that Gaussian centered at $x\in \mathbb{R}^d$. It is well-known that the
\emph{Kullback-Leibler divergence} between two such Gaussians recovers the scaled squared-distance
$t\lVert x-y\rVert^2$ between their respective centers. 
Then the function (on a suitable domain) given by
\begin{equation}
\label{eq:psigauss}    
\psi(\mu)=\sum_{i=1}^N c_i \exp(-\kl{\mu}{\mu_i})
\end{equation}
agrees with $\phi(x)$ for $\mu=\theta(x)$, and 
$\mu_i=\theta(x_i)$.

We will consider a version of \eqref{eq:psigauss} in the case of
finite nonvanishing distributions,
which will be given as elements of the positive
probability simplex
\begin{equation}
    \Delta^+_n=\left\{(p_1,...,p_n): p_1+\cdots+p_n=1,\ \mbox{$p_i> 0$ for all $i$}\right\}.
\end{equation}
Then for any pair of densities the Kullback-Leibler divergence
\begin{equation}
    \kl{p}{q}=\sum_{i=1}^n p_i\log(p_i/q_i),
\end{equation}
is always finite-valued, and is a non-symmetric distance in the
sense that $\kl{p}{q}\geq 0$, with equality if and only if $p=q$.
Our proposed potential function $\psi:\psimp\rightarrow \mathbb{R}$
can now be defined by
\begin{equation}
\label{eq:defpsi}    
\psi(p)=\log \sum_{i=1}^m  c_i 
\exp(-\kl{p}{q_{i\bullet}})
\end{equation}
for some coefficients $c_i>0$.
This function has high values on distributions
$p\in \Delta_n$ which are similar to many rows of 
$Q$, analogous to a statistical density estimator.

From the point of view of 
Topological Data Analysis, one might
consider studying the super-level set persistent homology of
$\psi$ directly to model the topological type of $Q$, regarded as a data set.
Computationally, doing this without applying the transport map would be even less tractable than doing so for Gaussian kernels, 
due to the high dimensionality of $\Delta^+_n$. However, 
the authors of \cite{carlsson2022search} recovered topological data of square stochastic matrices coming from Markov chains 
by restricting a related potential function to the low-dimensional skeleta of $\Delta_n$. 
Both this construction and $\psi$ have the property that they
lift discrete data from a Markov chain to a continuous space.

In Theorem \ref{thm:main} below, we prove a version of the main
result of \cite{carlsson2025kernel} for these functions,
which says that $\psi$ is $c$-convex with 
respect to the discrete KL-divergence $c(p,q)=\kl{p}{q}$, and we give
an explicit formula for the resulting 
transport map $T:\Delta^+_n\rightarrow \Delta_n^+$. We prove that $T$ determines a homotopy
equivalence between the super-level sets of $\psi$ and those of the conjugate function 
$\eta$, defined on the range of $T$. 
In Corollary \ref{cor:gaussians}, we describe how the topology of the 
Gaussian potential function 
is derived as a special case.

In order to apply the construction to datasets, we
describe a method for converting a point cloud
to the desired input form of a
positive stochastic matrix $Q=(q_{ij})$,
which is one for which $q_{i\bullet}\in \psimp$.
To do this, we begin by generating the transition matrix $P$ of a Markov chain whose states are identified with the points of the dataset, using the Stochastic Neighbors construction, which is the first step of the $t$-SNE algorithm \cite{hinton2008tsne}.
We then define $Q=P(t)$ to be the corresponding continuous time transition matrix at a fixed time $t$, using the matrix exponential. This effectively extends the local kernel functions to a discretized one at a larger scale, while remaining confined to the geometry of the point cloud instead of cutting across it as Gaussian kernels would.
Generally, the approach of encoding of point cloud using a random walk
is similar in spirit to diffusion maps \cite{coifman2006diffusion}.

In Sections \ref{sec:transport} and \ref{sec:proofs}, we present and prove Theorem \ref{thm:main} as well as Corollary \ref{cor:gaussians}.
In Section \ref{sec:implementations}, we explain a 
method for generating
the input data of a stochastic matrix and coefficient vector 
for our transport construction, using
continuous time Markov chains derived from point
clouds using stochastic neighbors \cite{hinton2008tsne}.
In Section \ref{sec:experiments}, we demonstrate the pipeline on
four examples: a noisy synthetic point cloud whose core circle is
recovered; uniform samples from a square, where the method avoids
the pinching artifacts of a default UMAP embedding; the COIL
rotating-object images \cite{nene1996coil100}, where resampling the
shape repairs the spurious homology classes produced by a standard
PCA projection; and the SYMSOL pose-estimation benchmark
\cite{wiersma2021symsol}, where we recover the topology of
$SO(3)/A_4$, the space of orientations of a tetrahedron modulo its
rotational symmetries, from $45{,}000$ images.

\subsection{Acknowledgments}

The authors gratefully acknowledge the support of ONR Grant N000142112208, ``Topological data analysis in optimization'', as well as the support of the Capital One Center for Responsible AI Decision Making in Finance.

\section{Transport construction}

\label{sec:transport}

We define the constructions discussed in the introduction, and formulate the statement of our main result.
We first begin with a definition from
Optimal Transportation Theory \cite{villani2008old}.
\begin{defn}
\label{def:cconv}
Let $c:\mathcal{X}\times \mathcal{Y}\rightarrow \mathbb{R}$
be a general \emph{cost function}
on two spaces $\mathcal{X},\mathcal{Y}$. We say that
a function $\psi:\mathcal{X}\rightarrow \mathbb{R}$
is $c$-\emph{convex} if there exists a function $\zeta:\mathcal{Y}\rightarrow \mathbb{R} \cup \{-\infty\}$ satisfying
\begin{equation}
\label{eq:defcconv}    
\psi(x)=\sup_{y\in \mathcal{Y}} (\zeta(y)-c(x,y))
\end{equation}
In this case, we have a particular choice of $\zeta$
called the \emph{conjugate function}
\begin{equation}
\label{eq:cconj}
\psi^c(y)=\inf_{x\in \mathcal{X}} (\psi(x)+c(x,y))
\end{equation}
which would satisfy $\psi^c(y)=\sup \zeta(y)$,
where the supremum is over all functions $\zeta$ satisfying
\eqref{eq:defcconv}. When $\psi$ is $c$-convex, a function $T:\mathcal{X}\rightarrow \mathcal{Y}$ is called a \emph{transport map}
if we have
\begin{equation}
\label{eq:transportcondition}    
\psi^c(T(x))-\psi(x)=c(x,T(x))
\end{equation}
for all $x\in \mathcal{X}$.
\end{defn}

\begin{figure}
    \centering
    \begin{subfigure}[b]{.54\linewidth}
        \includegraphics[scale=.15]{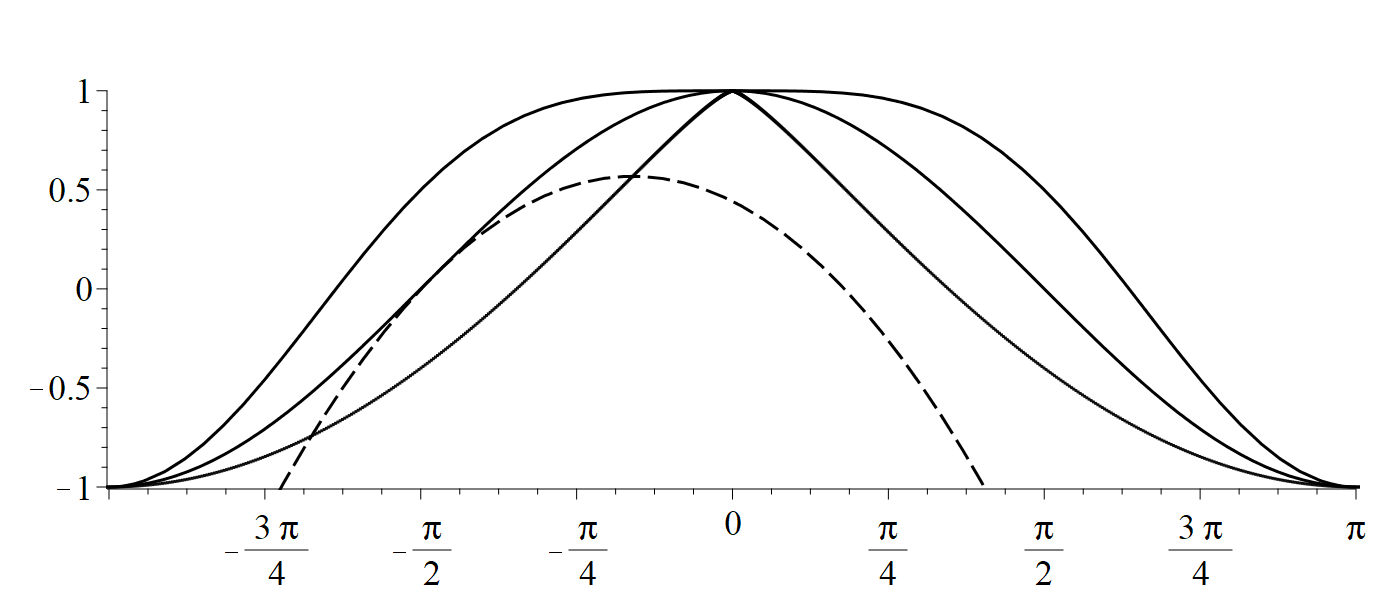}
    \end{subfigure}    
        \begin{subfigure}[b]{.45\linewidth}
        \includegraphics[scale=.14]{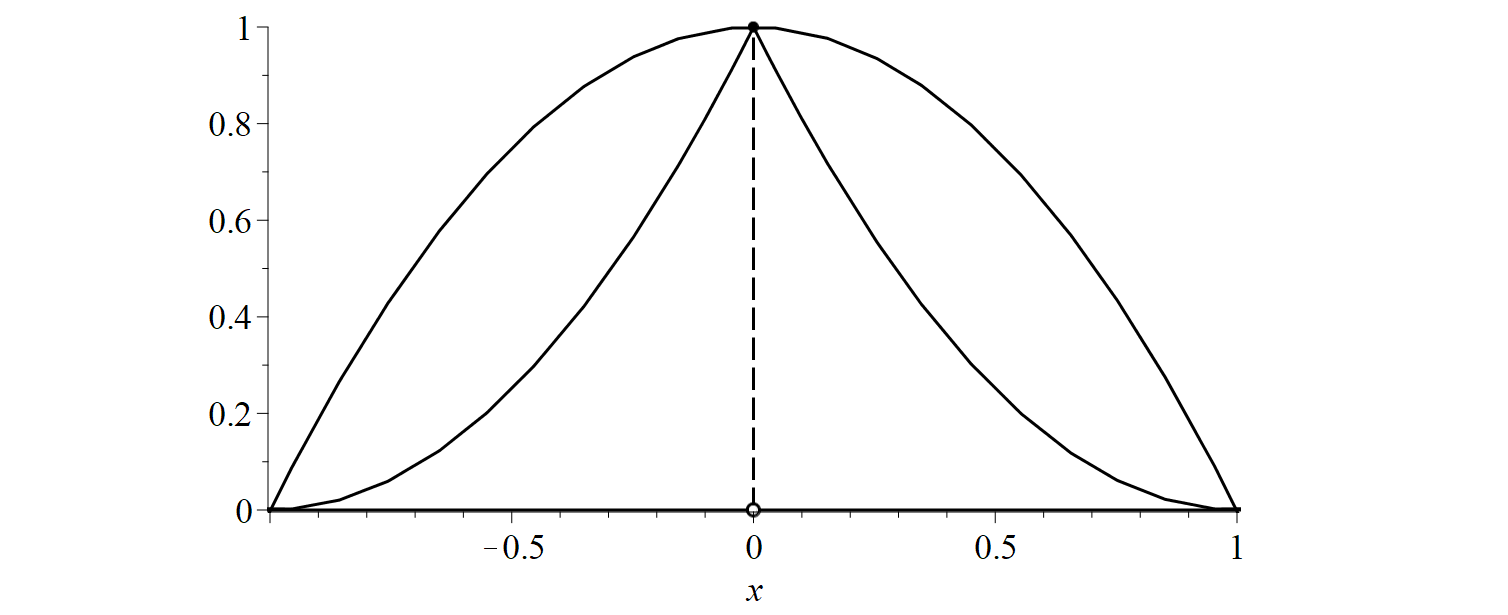}
    \end{subfigure}    
    \caption{Left: The function $\psi(x)=\cos(x)$ in the middle,
    surrounded by $\psi^c(T(x))$ and $\psi^c(x)$ for $c(x,y)=\lVert x-y\rVert^2/2$. Right: The function $y=1-x^2$ together with two potential dual functions $\psi^c$, and $\zeta$ given by the
    delta function at the origin.}
    \label{fig:cconv}
\end{figure}

In particular, for any $x$, the value $y=T(x)$
always realizes the supremum in \eqref{eq:defcconv}.
There may be more than one function $\zeta$ satisfying
\eqref{eq:defcconv}. For instance, if
$T$ is not surjective, we may always take
\begin{equation}
    \label{eq:defzeta}
    \zeta(y)=\begin{cases}
        \psi^c(y)
        & \mbox{$T(x)=y$ for some $x$} \\
        -\infty & \mbox{otherwise},
    \end{cases}
\end{equation}
which would be different from $\psi^c$.
While we do not have a closed expression 
for $\zeta$ or $\psi^c$ at a given point $y$, we can always calculate
$\zeta(T(x))=\psi^c(T(x))$ from \eqref{eq:transportcondition}.

Figure \ref{fig:cconv} 
illustrates how these objects are related.
First, the function 
    $\psi(x)=\cos(x)$, which is $c$-convex for $c(x,y)=\lVert x-y\rVert^2/2$ is shown. The transport map
    $T(x)=x+d/dx \cos(x)$ is the one that sends the contact point of the upside-down parabola (dashed) to the point $y$ corresponding to the peak. The lower graph is the conjugate function $\psi^c$ which is traced out by the peaks, while the outer one is 
    $\psi^c(T(x))=\psi(x)+c(x,T(x))$. On the right, we have
    a potential function given by $\psi(x)=1-x^2$ on the restricted domain $[-1,1]$. This function has more than one possible function $\zeta$ in Definition \ref{def:cconv}, namely the  function $\psi^c$ shown below it and the delta function at the origin, which is the function defined by \eqref{eq:defzeta}.

In our setting, we have $\mathcal{X}=\Delta^+_n$, 
$\mathcal{Y}=\psimp$,
$c(p,q)=\kl{p}{q}$,
and $\psi$ is the function from \eqref{eq:defpsi}.
We will show that $\psi$ is $c$-convex, and that there is a unique transport map defined as follows.
For any vector
$a=(a_j)_{j=1}^n$ with positive entries, we will use the notation
\begin{equation}
\label{eq:normprob}
[a]=[a_j]_{j=1}^n=p \in \Delta^+_n,\ \ p_j=
a_j/(a_1+\cdots+a_n).
\end{equation}
We then define
\begin{equation}
\label{eq:kltransport}    
T(p)=[q_{1j}^{u_1} \cdots q_{mj}^{u_m}]_{j=1}^n,\ \ 
u=[c_i \exp(-c(p,q_{i\bullet}))]_{i=1}^m \in \Delta_m^+.
\end{equation}

%
%

We may now state our theorem. Let us denote
$\mathcal{P}=\mathcal{Q}=\Delta_n^+$ in place of $\mathcal{X}$
and $\mathcal{Y}$, and
consider the super-level set
\begin{equation}
\label{eq:psilev}
\PP(a)=\left\{p\in \psimp : \psi(p)\geq a\right\}.
\end{equation}
We define 
\begin{equation}
\label{eq:defqq}
    \QQ(a)
    =\left\{T(p):\psi(p)+c(p,T(p))\geq a\right\},
\end{equation}
which agrees with the super-level set
$\left\{q:\zeta(q)\geq a\right\}$ when 
when $T$ is a transport map and
$\zeta$
is defined by \eqref{eq:defzeta}. Then we have
\begin{thm}
\label{thm:main}
Let $\psi$ be defined by \eqref{eq:defpsi}, 
let $T$ be defined by \eqref{eq:kltransport},
and let $\PP(a)$ and $\QQ(a)$ be the sets from 
\eqref{eq:psilev} and \eqref{eq:defqq}. Then
\begin{enumerate}
\item 
We have that $\psi$ is $c$-convex
for $c(p,q)=\kl{p}{q}$,
and that $T$ is the unique transport map
associated to $\psi$.
\item \label{item:homotopy} We have inclusions
$T(\mathcal{P}(a))\subset \qspace(a)\subset \pspace(a)$ for all 
$a\in \mathbb{R}$, which induce homotopy 
equivalences
$T(\mathcal{P}(a))\simeq \QQ(a) \simeq \PP(a)$.
The inverse equivalences are all determined by
the restriction of $T$.
\end{enumerate}
    
\end{thm}

As in 
\cite{carlsson2025kernel,carlsson2025cosim},
we interpret the nested family of regions
$\QQ(a)$ as a continuous version of
an \emph{alpha shape} \cite{edelsbrunner1983shape}, and the equivalence $\PP(a)\simeq \QQ(a)$
as a counterpart to Edelsbrunner's theorem
\cite{edelsbrunner1995union}, which identifies its homotopy type with that
of the power diagram by using an explicit 
deformation retraction. 
The reason for making this connection is that
we can reformulate \eqref{eq:defcconv} as
\begin{equation}
\label{eq:cconvdiscs}    
\PP(a)=\bigcup_{q\in \QQ(a)} B_{c}(q;\zeta(q)-a),\ \ 
B_c(q;r)=\left\{p\in \Delta^+_n: c(p,q)\leq r\right\},
\end{equation}
which is precisely the relationship between a power diagram
and the weighted sites in the alpha complex with
$-\zeta$ in place of the weight map $\alpha$. 

We also remark that the statement of item \ref{item:homotopy} cannot 
be interpreted as continuum limit of nerve theorems, which would not say anything about the topology of a continuous set of 
centers in a covering. 
Additionally, Edelsbrunner's deformation retract would be from
$\PP(a)$ down to $\QQ(a)$ under the analogy of \eqref{eq:cconvdiscs}, whereas ours is based on another space given by
$\mathcal{R}(a)=T^{-1} \mathcal{Q}(a)$. 
Finding a deformation retract from $\PP(a)$ to $\QQ(a)$ involving
the transport map appears difficult due to the fact that the restriction of $T$ to $\PP(a)$ does not surject onto $\QQ(a)$, 
$T$ does not act as the identity on $\QQ(a)$, and
testing membership $q\in\QQ(a)$ does not have a closed form without
expressing $q=T(p)$ for some $p$.

\section{Proof of Theorem \ref{thm:main}}

\label{sec:proofs}

In this section we prove Theorem \ref{thm:main}, and then show in
Section \ref{sec:gaussians} that Gaussian kernel density estimators
arise as a special case.

\subsection{Proof of c-convexity}

Let $Q$ be an $m\times n$ positive stochastic matrix, and let $\psi$, and $T$ be given as in the previous section for some coefficients $c_i>0$, and let $\mathcal{P}(a),\mathcal{Q}(a)$ 
be the super-level sets.

\begin{lemma}
\label{lem:istransport}    
Let $\bar{q}=T(\bar{p})$ for some $\bar{p}\in \Delta_n$.
Then $p=\bar{p}$ is a global
minimizer of the expression
$\psi(p)+c(p,\bar{q})$.

\end{lemma}

\begin{proof}

For fixed $\bar{q}=T(\bar{p})$,
we will show that the function $F(p)=\psi(p)+c(p,\bar{q})$
is convex, and that its 
gradient is zero at $p=\bar{p}$ for all $p\in \Delta^+_n$.

The gradient will be realized by regarding
$F(p)$ as a function on 
$\mathbb{R}_{>0}^n$ whose tangent space is
$\mathbb{R}^n$, and then taking the image in
$\mathbb{R}^n/\langle {\mathbf 1}\rangle=T^*_p \Delta^+_n$.
First, we calculate
\[
\frac{\partial}{\partial p_j} \psi(p)=
\frac{\partial}{\partial p_j} \log \sum_{k} c_k \exp(-c(p,q_{k\bullet}))=
\sum_{k} u_k (\log(p_j/q_{kj})+1)
\]
where $u_k$ are the weights in \eqref{eq:kltransport}.
On the other hand, we have
\[\frac{\partial}{\partial p_j} \kl{p}{\bar{q}}=\log(p_j)-s^{-1}\log(q_j+1)=
\log(p_j)-\sum_{k} 
\bar{u}_k\log(q_{kj})+C\]
where $\bar{u}_k$ are the weights
at $p=\bar{p}$, and 
$C$ is a constant that does not depend on $j$. Then the two sum to zero as elements of 
$\mathbb{R}^n/\langle \mathbf{1}\rangle$ when $p=\bar{p}$, so that $u_k=\bar{u}_k$,
showing that $\nabla_pF=0$ at that value.

To establish the convexity, we take another partial derivative to obtain the Hessian matrix $H=(H_{ij})$, where 
\[H_{ij}=\frac{\partial}{\partial p_i} 
\frac{\partial}{\partial p_j} 
F(p)=\sum_k (\log(p_j/p_{kj})+1) 
\frac{\partial}{\partial p_i} u_k=\]
\[\sum_k u_k (\log(p_i/p_{ki})+1)
(\log(p_j/p_{kj})+1) -\]
\[\left(\sum_{k} u_k (\log(p_i/p_{ki})+1)\right)
\left(\sum_{k} u_k (\log(p_j/p_{kj})+1)\right).\]
Note that the result does not 
depend on $\bar{q}$, and that
we have canceled two 
delta function
factors $\delta_{ij}$.

It suffices to show that $H$ is positive semi-definite.
For this, take an arbitrary dot product 
\[v^tHv=\left(\sum_{k} u_k x^2_k\right)-\left(\sum_k u_k x_k\right)\left(\sum_k u_k x_k\right)\]
for any vector $v\in \mathbb{R}^n$, where
\[x_k=\sum_{i} v_i \left(\log(p_i/p_{ki})+1\right)\]
Positive semi-definiteness now follows from 
the Cauchy-Schwarz inequality, noting that the weights $u_k$ are positive and sum to one.
\end{proof}

Then we have the following
\begin{cor}
\label{cor:cconv}
The potential $\psi$ is $c$-convex with unique transport map $T$, 
and the above functions satisfy
\begin{equation}
\label{eq:ineq_allthree}    
\zeta(p)\leq \psi^c(p) \leq \psi(p)\leq \zeta(T(p)).
\end{equation}
\end{cor}

\begin{proof}

The argument is similar to the one explained after the differential criteria for $c$-convexity in \cite{villani2008old}:
by Lemma \ref{lem:istransport},
we have
\begin{equation}
\label{eq:cconv_equiv}    
\psi(p)+c(p,\bar{q})\geq
\psi(\bar{p})+c(\bar{p},\bar{q})
\end{equation}
for all $p\in \psimp$. In particular,
we have that
$\psi^c(\bar{q})=\psi(\bar{p})+c(\bar{p},\bar{q})$
whenever $\bar{q}=T(\bar{p})$,
On the other hand, we always have 
$\psi(p)\geq \sup_{q} (\psi^c(q)-c(p,q))$, since
\[\psi^c(q)-c(p,q)=\inf_{p'} (\psi(p')+c(p',q))-c(p,q)
\leq \psi(p)\]
for all $p,q$. The previous statement shows that we have equality, so that $\zeta$ satisfies \eqref{eq:defcconv}.

This also establishes \eqref{eq:ineq_allthree},
where the final inequality follows from the definition of $\zeta$ and because $c(p,q)\geq 0$. 

The uniqueness of $T$ follows since
$c(p,q)$ can be seen to satisfy a property called the 
\emph{twist condition}
(see \cite{villani2008old}, equation (10.4)), which says that the function
$q\mapsto \nabla_p c(p,q)$ into the tangent space at $p$ is injective. In this case 
there can only be one
$q$ for which $\nabla_p(p)+\nabla_pc(p,q)=0$, and we must have
$T(p)=q$ in \eqref{eq:transportcondition}.

\end{proof}

\subsection{Proof of homotopy equivalence}

We now prove the item \ref{item:homotopy} of Theorem \ref{thm:main} 
using an explicit deformation retraction. We start with more
lemmas.

\begin{lemma}
\label{lem:flow}    
Suppose that $q=T(p)$, and let
$p_s=[p^{1-s}q^s]$. Then we have
that $\psi(p)\leq \psi(p_s)$ for $0\leq s \leq 1$.
\end{lemma}

\begin{proof}

By slightly extending the calculations from 
Lemma \ref{lem:istransport}, we can see that
$\psi$ is $c$-convex for the cost function 
$c_s(p,q)=s^{-1}\kl{p}{q}$ for any $0<s\leq 1$,
and that the transport map is $T_s(p)=p_s$.
Inserting $T=T_s$ 
and the corresponding definition of $\zeta_s$
into \eqref{eq:ineq_allthree}, we find that
$\psi(p)\leq \zeta_s(p_s)\leq \psi(p_s)$.

\end{proof}

Given $Q=(q_{ij})$ and $c=(c_i)$, let us define the transformed data
$\tilde{Q}=(\tilde{q}_{ij})$ and $\tilde{c}=(\tilde{c}_i)$ 
to be the $n\times m$ normalized transpose matrix and $n$-dimensional coefficient vector given by
\begin{equation}
\label{def:tildestoch}
    \tilde{q}_{ji}=c_i q_{ij}/\tilde{c_j},\ \ 
    \tilde{c}_j=c_1q_{1j}+\cdots+c_mq_{mj}
\end{equation}
Inserting these into \eqref{eq:defpsi}, 
we then obtain another potential function on the 
positive $m$-dimension probability simplex 
$\tilde{\psi}:\Delta^+_m\rightarrow \mathbb{R}$, 
as well as the transport map $\tilde{T}:\Delta_m\rightarrow \Delta^+_m$.

We have the following relations between these maps, 
which will be useful for proving 
both the homotopy statements, and a relation with the Gaussian case in Section \ref{sec:gaussians}:
\begin{lemma}
\label{lem:dualpsi}    
Let $u$ be the weight vector from
\eqref{eq:kltransport} associated to $p$,
and let $v$ be the one associated to $q=T(p)$. Then
we have
\begin{equation}
\tilde{\psi}(u)=\psi(p)+c(p,T(p)),\ \ 
\tilde{T}(u)=v.
\end{equation}
\end{lemma}
\begin{proof}
This can be computed directly by rearranging sums using the definition
\eqref{eq:defpsi} of $\psi$ and the weights \eqref{eq:kltransport}. 
\end{proof}

We can now proceed to the homotopy statement from Theorem \ref{thm:main}. Following \cite{carlsson2025kernel,carlsson2025cosim}, the best approach turns out not to be to find a deformation retraction from $\PP(a)$ to $\QQ(a)$, 
since $T$ does not act by the
identity on $\QQ(a)$, and $\zeta$ does not have a closed expression. Instead, we introduce a third space
\begin{equation}
\label{eq:defR}    
\RR(a)=T^{-1} \QQ(a)=\left\{p: \psi(p)+c(p,T(p))\geq a\right\} \supset \PP(a)\supset \QQ(a) 
\end{equation}
We will produce deformation retracts from this space to 
both $\PP(a)$ and $\QQ(a)$, showing that all three are
homotopy equivalent.

Both deformation retracts will be based on the following.
\begin{lemma}
\label{lem:path}
Let $q=T(p)$ for $p\in \RR(a)$. Then we have that
$(1-s)p+sq\in \RR(a)$ for all $0\leq s \leq 1$.
\end{lemma}

\begin{proof}

Let $u$, $u_s$ and $v$ be the weight vectors 
from \eqref{eq:kltransport} associated to 
$p$, $p_s$, and $q$ respectively,
which we easily see satisfy
$u_s=[u_i^{1-s}v_i^s]_{i=1}^m$. 
By Lemma \ref{lem:flow} applied to $\tilde{\psi}$, 
we see that
$\tilde{\psi}(u)\leq \tilde{\psi}(u_s)$ for all $s$ in the range. The statement then follows 
from Lemma \ref{lem:dualpsi}.
    
\end{proof}

We can now prove Theorem \ref{thm:main}:
\begin{proof}
The $c$-convexity statement is given 
in Corollary \ref{cor:cconv}. For the second, 
we first prove that 
$\RR(a)$ is homotopy equivalent to $\QQ(a)$.
In this case we have that 
$T:\RR(a)\rightarrow \QQ(a)$ is surjective
by \eqref{eq:defqq}, so that the fibers of
$T$ are nonempty. Then by the proof of Lemma \ref{lem:istransport},
they are convex and so contractible, showing that $T$
is a homotopy equivalence. By Lemma \ref{lem:path}
we see that $(1-s)p+sT(p)$ remains in $\RR(a)$, and gives a homotopy between the identity map on $\RR(a)$ and the composition of $\iota T$ where $\iota:\QQ(a)\rightarrow \RR(a)$ is the inclusion, 
showing that the inverse homotopy is 
induced by $\iota$.

To show that $\RR(a)\sim \PP(a)$, we define a deformation retraction $H: \RR(a)\times [0,1]\rightarrow \RR(a)$ from down to $\PP(a)$. Given $p\in \RR(a)$ with $q=T(p)$, let
$s'\in [0,1]$ be the smallest value with the property that
$(1-s')p+s'q\in B_{c}(q;a-b)$ where 
$b=\zeta(q)=\psi(a)+c(p,q)$, and $B_c(q;r)$ is the ball defined
in \eqref{eq:cconvdiscs}.
We then define
\begin{equation}
H(p,s)=\begin{cases}
(1-s)p+sq & s\leq s'    \\
(1-s')p+s'q & s>s'.
\end{cases}    
\end{equation}
By Lemma \ref{lem:path}, this remains in $\RR(a)$.
Since $\PP(a)$ is the union of balls \eqref{eq:cconvdiscs}, 
we find that $s'=0$ for $p\in \PP(a)$, so that $H(p,s)=p$ for all 
$s$, showing that it is a deformation retraction.

The statement that $T(\mathcal{P}(a))\simeq \mathcal{P}(a)$ follows similarly, using Lemma \ref{lem:flow} and the flow $p_s$
in place of the one in Lemma \ref{lem:path}.

\end{proof}

\subsection{Relation to the Gaussian case}

\label{sec:gaussians}

Let $\mathcal{D}=\{x_i\}_{i=1}^n\subset \mathbb{R}^d$ be a point cloud, and define $\phi(x)=\log(f(x))$ be as in 
\eqref{eq:gaussiansum}
for some coefficients $a=(a_i)_{i=1}^n$ and 
fixed scale parameter $t>0$. 
In the language of usual Legendre duality, it was shown 
in \cite{carlsson2025kernel} that $\phi$ is
$c$-convex for the cost function $c(x,y)=t\lVert x-y\rVert^2$, and that a version of 
Theorem \ref{thm:main}
was satisfied, connecting $\mathcal{X}(a)$ to a dual region $\mathcal{Y}(a)\subset \bar{\Delta}_n^+$, where 
$\bar{\Delta}^+_n=\chull(\mathcal{D})^+$ denotes the interior of 
the convex hull. 
The counterpart of the 
transport map is given by
$\Phi:\mathbb{R}^d\rightarrow \bar{\Delta}_n^+$, where
\begin{equation}
\label{eq:gausstransport}    
\Phi(x)=\sum_{i=1}^n \theta(x)_i x_i,\ \ 
\theta(x)=[\exp(-t\lVert x-x_j\rVert^2)]_{j=1}^n.
\end{equation}
The dual function $\phi^*$ can be written as 
$-\alpha$, where $\alpha$ was interpreted as the weight map in the context of alpha shapes.

The following result connects this to our KL-based potential functions.
\begin{cor}
\label{cor:gaussians}  
Let $\phi(x)=\log(f(x))$ and $\mathcal{X}(a)=\phi^{-1}[a,\infty)$ be as above. Then $\mathcal{X}(a)$
is homotopy equivalent to 
$\mathcal{P}(a)=\psi^{-1} [a,\infty)\subset \Delta^+_n$ where
\begin{equation}
\label{eq:psigaussintegral}
    \psi(p)= \log \int_{\mathbb{R}^d} 
    f(x)\exp(-\kl{p}{\theta(x)})  dx.
    \end{equation}
\end{cor}

We start with two lemmas, the first being
an analog of Lemma \ref{lem:dualpsi}.
\begin{lemma}
\label{lem:conjugategaussians}    
The transport map for $\eqref{eq:psigaussintegral}$ is given by
\begin{equation}
\label{eq:transportgaussian}
T(p)=\theta(\Phi'(p)),\ \ \Phi'(p)=p_1x_1+\cdots+p_nx_n,
\end{equation}
and we have a relation
\begin{equation}
\label{eq:conjgauss}    
\phi^*(\Phi(x))=\phi(x)+t\lVert x-y\rVert^2=\psi(\theta(x)),\ \ T(\theta(x))=\theta(\Phi(x)).
\end{equation}
\end{lemma}
\begin{proof}
These relations can be checked directly, by comparing gradient vectors as in Lemma \ref{lem:istransport}, or by taking the 
limit of transport maps by writing \eqref{eq:transportgaussian} as a Riemann sum. In the second case,
we would calculate
$T(p)=[ \exp(- r_j)]_{j=1}^n$ where
\[r_j=-\int f(x) \exp(-\kl{p}{\theta(x)}) \log(\theta(x)_j) dx.\]
We then obtain \eqref{eq:transportgaussian}
by inserting $\kl{q}{\theta(x)} = t\lVert z-\Phi'(p)\rVert^2+c$, which has the effect of shifting the centers of the Gaussians defining $f(x)$ towards the mean. The rest can be deduced by taking the limit of Lemma \ref{lem:dualpsi}, realizing that $u=\theta(x)$ plays the role of the weights in \eqref{eq:kltransport}.    
\end{proof}

\begin{lemma}
\label{lem:maxpsi}    
Let $\psi$ and $T$ be as in \eqref{eq:psigaussintegral} and \eqref{eq:transportgaussian}, and suppose that $q$ is in the image
of $T$. 
Then there exists a point $x\in \mathbb{R}^d$
such that $T(\theta(x))=q$, and
$\psi(\theta(x))\geq \psi(p)$ for all 
$p \in T^{-1}(\{q\})$.
\end{lemma}

\begin{proof}

We may assume without loss of generality that the affine
span of $\mathcal{D}$ is all of $\mathbb{R}^d$, so that they do not lie in a proper affine subspace, which in particular means that
$\theta$ is injective.

Suppose $q$ is in the image of $T$, say $q=T(p)$.
Then Lemma \ref{lem:conjugategaussians} says we have a 
commuting square
\begin{equation}
\label{eq:commutingtransport}
\begin{tikzcd}
\mathbb{R}^d \arrow[r,"\theta"] \arrow[d,"\Phi"']
  & \Delta_n^+ \arrow[d,"T"] \arrow[dl,"\Phi'"] \\
\bar{\Delta}_n^+ \arrow[r,"\theta"']
  & \Delta_n^+
\end{tikzcd}
\end{equation}
which shows that $q=\theta(y)$ for some $y\in \mathbb{R}^d$.
It was shown in \cite{carlsson2025kernel} that
every point $y \in \chull(\mathcal{D})$ is given by
$y=\Phi(x)$, so that $q=\theta(\Phi(x))=T(\theta(x))$.
Moreover, $x$ is unique assuming the affine spanning condition
of $\mathcal{D}$.

We will show that $\psi$ is maximized at $p=\theta(x)$
over all points in $T^{-1}(\{q\})$. Since
$\psi(p)+\kl{p}{q}$ is constant on that set, this amounts
to saying that $\kl{p}{q}$ is minimized at that value.

The fiber $T^{-1}(\{q\})$ can be described by the
condition that $Ap=y$ where
$A$ is the $d\times n$ matrix
whose $i$th column is $x_i$.
For any $q$, the minimizer of the
function $\kl{p}{q}$ subject to this constraint is given by
\begin{equation}
\label{eq:infoproj}    
p=[q_i \exp(-(A^t\lambda(q))_i)]_{i=1}^n
\end{equation}
for some vector $\lambda(q)\in \mathbb{R}^d$. 
If $q=\theta(y)$, then the form of \eqref{eq:infoproj} shows that 
$p$ is in the image of $\theta$, and therefore $p=\theta(x)$,
since $x$ is the unique point with $\Phi(x)=\Phi'(p)$.

\end{proof}

We can now prove the corollary.

\begin{proof}

By Lemma \ref{lem:conjugategaussians} 
we have $\phi(x)+t\lVert x-\Phi(x)\rVert^2=\psi(\theta(x))$, so that
$\mathcal{Z}(a)=\theta^{-1} (\mathcal{P}(a))$ where $\mathcal{Z}(a)$ is the super-level set of the left hand side. In the proof of the main result in \cite{carlsson2025kernel}, the role of 
$\mathcal{Z}(a)$
parallels that of $\mathcal{R}(a)$ in the present paper,
and it was shown that $\mathcal{Z}(a) \simeq \mathcal{X}(a)$.
It remains to show that
$\theta:{\mathcal{Z}(a)}\rightarrow \mathcal{P}(a)$ is a 
homotopy equivalence. The statement of Theorem \ref{thm:main}
applies to $T$, since it has the same form with a convergent 
integral in place of a finite sum, so we find that
it induces an equivalence 
$\mathcal{P}(a) \simeq T(\mathcal{P}(a))$
by item \ref{item:homotopy}.

We then must show that $T\theta : \mathcal{Z}(a)\rightarrow T(\mathcal{P}(a))$ is an equivalence.
Lemma \ref{lem:maxpsi} shows that this map is surjective,
so it suffices to check that it has contractible fibers.
For this we use the other sides of \eqref{eq:commutingtransport}.
Another result in \cite{carlsson2025kernel} shows that the fibers
of $\Phi$ are contractible, and this is also true for $\theta$
(in fact, both fibers are a point when the affine span of
$\mathcal{D}$ is $\mathbb{R}^d$).

\end{proof}

\begin{figure}
    \centering
\begin{subfigure}[b]{.32\linewidth}
\fbox{\includegraphics[scale=.13]{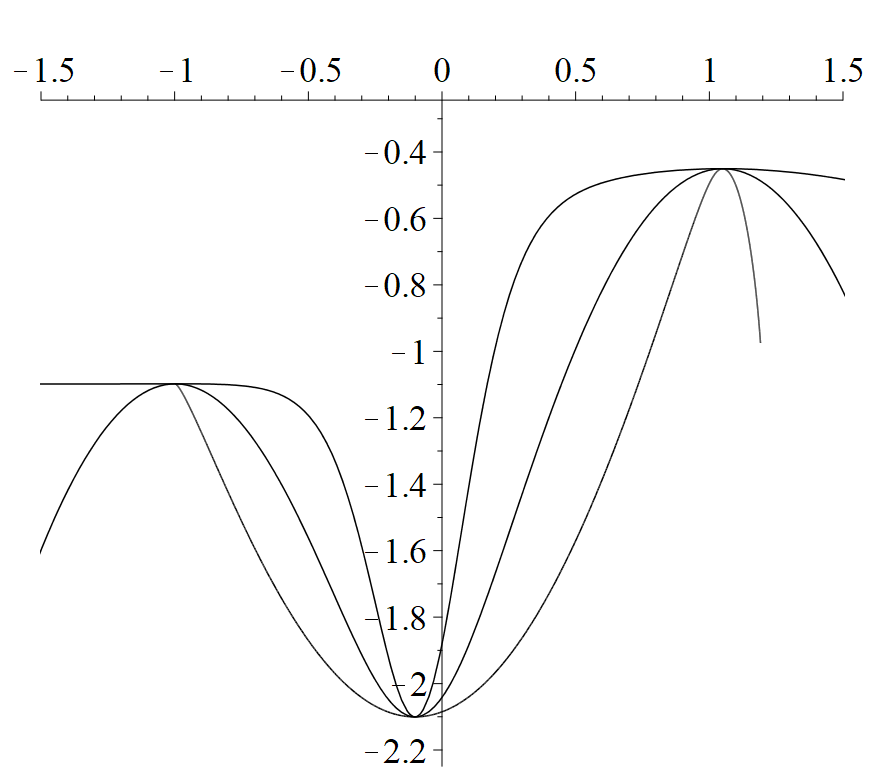}}
\end{subfigure}
\begin{subfigure}[b]{.32\linewidth}
\fbox{\includegraphics[scale=.13]{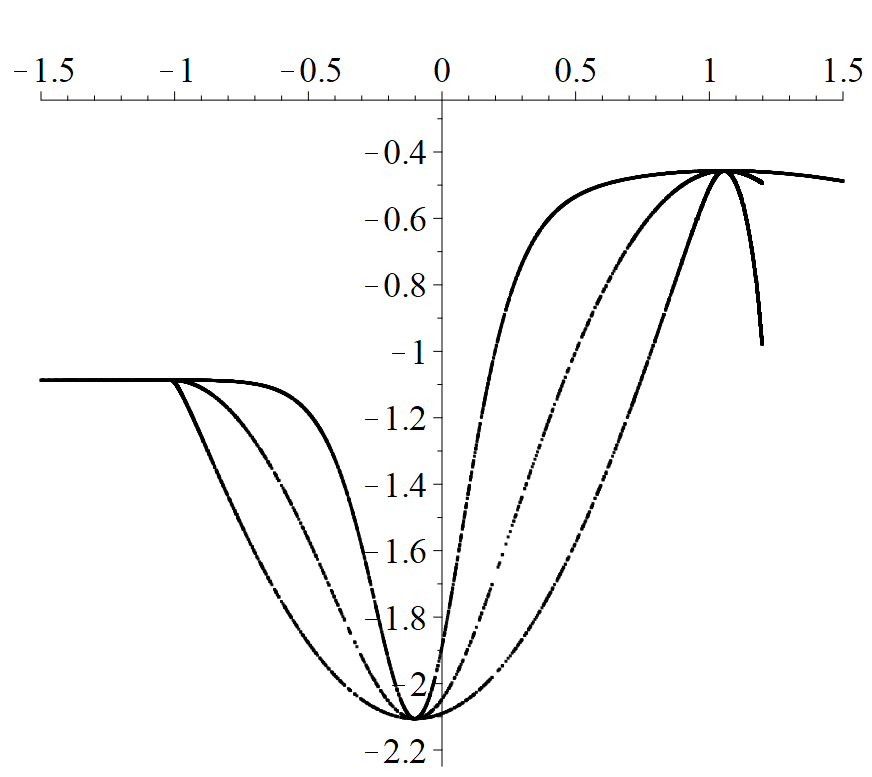}}
\end{subfigure}
\begin{subfigure}[b]{.32\linewidth}
\fbox{\includegraphics[scale=.216]{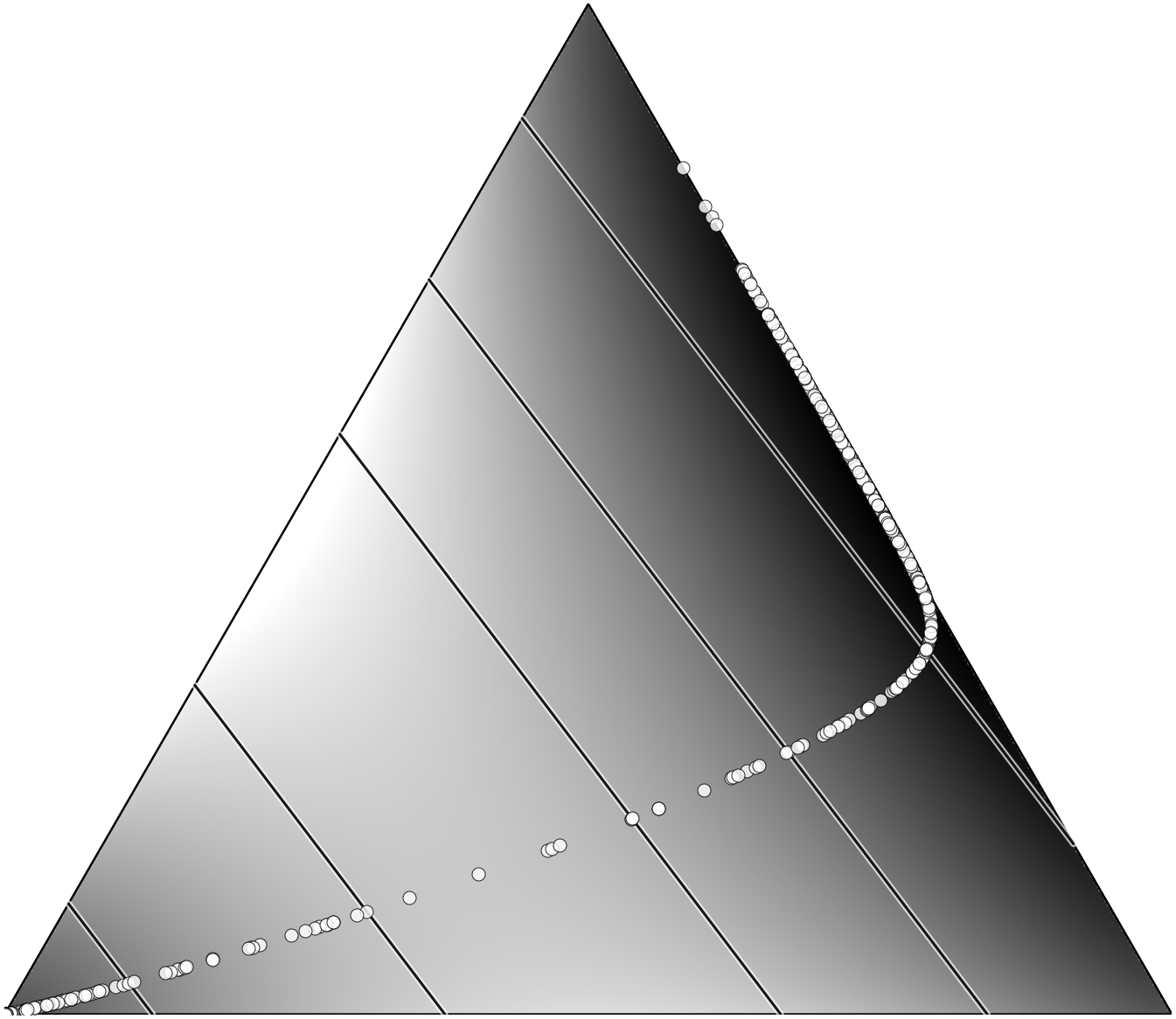}}
\end{subfigure}
    \caption{Left: The functions 
    $\phi^*(p(x))$, $\phi(x)$, and $\phi^*$ associated to the points $\{-1,.9,1.2\}\in \mathbb{R}^1$ from the Gaussian case with $a_i=1/3$ and $t=2.0$. Middle: many pairs of the forms $(x,\psi(\theta(x)))$,
    $(\Phi(x),\zeta(T(\theta(x))))$, and $(\Phi(x),\psi(\theta(x)))$ representing samples from the three spaces from item \ref{item:homotopy} of Theorem \ref{thm:main} under $\theta$. Right: The images of samples under $\theta$ in 
    $\Delta^+_3$, with $\psi$ values represented by grayscale intensity.}
    \label{fig:gaussians}
\end{figure}

These statements are illustrated in the one-dimensional case
in Figure \ref{fig:gaussians}. The top curves in the two left
frames are represented by 
$\phi^*(\Phi(x))=\phi(x)+t\lVert x-\Phi(x)\rVert^*$ and $\psi(\theta(x))$ represented by samples, which agree by
Lemma \ref{lem:conjugategaussians}. The lower curves also agree,
as they represent the graph of the conjugate function $\phi^*$
and the points $(\Phi(x),\psi(\theta(x)))$ which similarly correspond to each other. The middle functions are the graphs of
$\phi$ and the dual function $\psi^c$ 
under $\theta$, which
are not equal. 
On the right, we have the image of many points
under $\theta$ in $\Delta_3^+$, where the dark lines represent the
fibers $T^{-1}(\{q\})$. Lemma \ref{lem:maxpsi} says that the maximum value of $\psi$ (shown as darker grayscale values)
along those fibers occurs at those points in
the image of $\theta$.

\section{Practical implementations}
\label{sec:implementations}

We describe a method for applying this
construction to point clouds which are not well suited to Euclidean distances, 
by creating a random walk whose states 
are points based on stochastic neighbors
\cite{hinton2008tsne}.
We then pass this matrix to a
strictly positive continuous-time transition matrix $Q=P(t)$, which is stored using a low-rank form of its entries in negative-log coordinates. We use the transport construction to generate a condensed point
cloud in $\Delta_n^+$, referred to as the \emph{transport shape};
our implementation is called \textsc{AlphaKL}, in reference to the
alpha-shape analogy of Section \ref{sec:transport}.

\subsection{Stochastic neighbors random walks}

\label{sec:markov}

Starting with a point cloud
$\mathcal{D}=\{x_1,\ldots,x_n\}\subset\mathbb{R}^d$, we first generate a Markov chain using stochastic neighbors, which is the first step of the $t$-SNE
algorithm \cite{hinton2008tsne} defined
as follows. Given a vector of scales
$t=(t_i)_{i=1}^n$, let
$P=(p_{ij})$ be the transition matrix of a 
random walk whose $i$th row is
\begin{equation}
p_{i\bullet}=
\bigl[\exp(-t_i\lVert x_i-x_j\rVert^2)\bigr]_{j=1}^n,
\end{equation}
after normalizing the row to sum to one. 
We will also use the convention
of setting the diagonal entry $p_{ii}$ to
zero before normalization.
Given an entropy value $H>0$,
we then select the unique scale $t_i$ which arranges
that $p_{i\bullet}$ has entropy $H$. 
A typical value in our experiments is $H=2$,
and for efficiency we retain only
the largest $\lceil3e^H\rceil$ off-diagonal entries in each row and normalize
again.  Although squared Euclidean distance is displayed above, the same
construction can be done with respect to any
squared distance.

\subsection{Matrix exponentials}

\label{sec:matrixexps}

By exponentiating this stochastic-neighbor matrix, we determine a stochastic matrix $Q=P(t)$,
whose rows are discretized kernel functions
which move around the shape of a point
cloud. 
When the graph of nonzero entries of 
$P$ is strongly 
connected, the resulting matrix 
$P(t)$ is positive stochastic.
Otherwise, we may consider its strongly connected components.
We then define $\psi$ by \eqref{eq:defpsi} using the coefficients
$c_i=1$, though for other Markov chains
it may make sense to define $c_i=\pi_i$, where
$\pi$ is the stationary distribution.

For efficiency, one might expect to compute the matrix exponential by taking a possibly truncated spectral decomposition of $P$.  The problem with
this is that truncated SVD approximations can
introduce small negative entries into $P$ or its powers, which results in 
the cost vector determining the 
transport map becoming infinite.
Instead, we compute the matrix
exponential in logarithmic coordinates.  Rather than storing only
$Q=(q_{ij})$, we store
\begin{equation}
R=(r_{ij}),\ \ r_{ij}=-\log(q_{ij}).
\end{equation}
To compute this matrix directly, we apply the formula
\begin{equation}
P(t)=\exp(-t(\Id-P))=\exp(-t)\exp(tP),
\end{equation}
and observe that every term in the power series for $\exp(tP)$ is nonnegative.  
We may therefore evaluate
the power series 
without ever introducing negative coordinates,
thereby preserving longer range coefficients which
would otherwise be rounded to zero.

One issue with computing the power series directly is that
the complexity grows with $t$, as more terms are required 
to get close to convergence.
We have found in our code that 
exponential matrices of size a few thousand can
be comfortably computed in this way in a few seconds or minutes, but for sizes
$n>10000$, a
more sophisticated method should be developed.
In the example of Section \ref{sec:tetrahedra}
in which $n=45000$, we instead deal with this using
the coarse graining procedure of Section \ref{sec:coarse_graining}.

\subsection{Matrix factorizations}

\label{sec:factorizations}

For more than a few thousand states, it also
becomes impractical to store $R$
densely or to evaluate $\psi$ and $T$ using dense linear algebra. We deal with this by replacing $R$ 
by a rank-$r$ gauged factorization
\begin{equation}
R\sim R'=A\Sigma B+\mu{\bf 1}^t,
\end{equation}
where $A,\Sigma,B$ have dimensions $n\times r$, $r\times r$, and
$r\times n$, respectively, and ${\mathbf 1}$ is the vector of all ones. The column
vector $\mu$ is a gauge correction chosen so that the entrywise
exponential $Q'=(\exp(-r_{ij}'))$ is stochastic. This is efficient
because the main step in computing the cost vector
$(\kl{p}{q_{i\bullet}})_{i=1}^n$, and therefore the potential and
transport map, is multiplying $R$ on the right by $p$, regarded as a
column vector. With the factorization above, the factors can simply be
multiplied in order.

The factorization is obtained by doubly centering $R$, as in
multidimensional scaling (MDS), applying a truncated singular value
decomposition to the centered matrix, and then correcting the row gauge
$\mu$ so that $Q'$ is stochastic as indicated above. 
In fact, when
$Q$ is given by Gaussian kernels as in Section \ref{sec:gaussians},
this construction precisely recovers classical MDS: the centered
negative logarithms differ from squared Euclidean distances only by row
and column gauge terms, so the SVD recovers the original low-dimensional
coordinate factors.
For very large matrices, we have found that a randomized SVD is effective
\cite{halkl2011randomized}, although we did not use this in our examples.

This factorization procedure also determines
a useful Euclidean embedding method: the leading
columns of $A\Sigma^{1/2}$ give coordinates for the rows of $Q$, in the same
way that leading principal components give visualization coordinates in
PCA. This is used to create low-dimensional embeddings for visualization as well as higher-dimensional embeddings 
to generate alpha complexes as explained in Section
\ref{sec:alphacomplex}.

\subsection{Sampling using the transport map}

\label{sec:sampling}

We have defined $\psi$ in Section \ref{sec:matrixexps},
and therefore the transport map and dual function
$T,\zeta$.
There are now two possible definitions of the shape,
which are the filtered families $\mathcal{Q}(a)$, or
$T(\mathcal{P}(a))$.
Theorem \ref{thm:main} shows that both are homotopy equivalent, though Lemma 
\ref{lem:conjugategaussians}
shows that $T(\mathcal{P}(a))$ may be the correct choice when $Q$ is derived
from a kernel function on a point cloud.
We present methods for
sampling from both.

In either case,
the resulting collection of points will be recorded as the rows of an $M\times n$ matrix denoted $\tilde{Q}=(q_{i\bullet})_{i=1}^M$, together with a vector $b=(b_i)_{i=1}^M$ of corresponding filtration values. To sample from $\mathcal{Q}(a)$, 
we would need to 
generate sampled points 
$\{\tilde{p}_i\}_{i=1}^M\in \Delta^+_n$,
and let 
$\tilde{q}_{i}=T(\tilde{p}_i)$, and
$b_i=\zeta(T(\tilde{p}_i))$.
To sample from $T(\mathcal{P}(a))$,
we would instead take $b_i=\psi(\tilde{p}_i)$, with the same choice of $\tilde{q}_i$.

For both spaces we are required to
generate the input points
$\tilde{p}_i\in\Delta_n^+$. We will use two methods for this:
\\
\noindent {\bf Method A.}
A simple possibility is to simply let $\tilde{p}_i=q_{i\bullet}$ be the $i$th row of $Q$.
This choice always
produces exactly $M=n$ transported points and can 
leave large portions of the
shape unexplored when the original data set is small,
as is the case COIL example below,
which has only $72$ images.
\\
\noindent {\bf Method B}.
The second method
involves extending the sampling method of \cite{carlsson2025kernel}.
For this we first use the left logarithmic MDS coordinates from Section \ref{sec:factorizations} to convert
the rows of $Q$ into a point cloud 
$\mathcal{D}=\{x_i\}_{i=1}^n
\subset \mathbb{R}^r$, where $r$ is the rank of the
factorization. We then sample from the corresponding
underlying distribution of $f(x)$ as in Section 
\ref{sec:gaussians} using the scale parameter $t=1$
to obtain a new point cloud $\tilde{\mathcal{D}}=\{\tilde{x}_i\}_{i=1}^M$ for $M$ large enough, and define the $i$th
sampled distribution by $\tilde{p}_i=\theta(\tilde{x}_i)$, 
where $\theta$ is defined 
as in \eqref{eq:gausstransport}.

\subsection{Euclidean embeddings and alpha complexes}
\label{sec:alphacomplex}

Given a matrix of samples $\tilde{Q}$, we
create a Euclidean embedding by applying the MDS method
from Section \ref{sec:factorizations} to the
matrix $\tilde{R}$ associated to $\tilde{Q}$ in logarithmic
coordinates, noting that the answer may not be square.
Taking the first two or three coordinates produces
an immediate low-dimensional embedding of the points, which may be colored according to the weights with
$b_i\geq a$. Visually,
this embedding has some similarities to 
UMAP and t-SNE in some cases, 
but can been seen to better respect the 
underlying geometry in others,
see Figure \ref{fig:square} below.

Using the MDS as above but in higher dimensions such as the full rank $r$ of the factorization also makes it possible
to generate \emph{alpha complexes} from the data,
which are simplicial complexes associated to a Euclidean point cloud based on a distance-restricted
version of the Delaunay triangulation
\cite{edelsbrunner2010computational}.
They have a filtered version which 
can be used to compute homological invariants of point clouds, but because they involve a minimal number of simplices compared to other constructions, they
also determine refined, visually appealing geometric models \cite{edelsbrunner1983molecule}.

To construct the alpha complex from the MDS embedding
of $\tilde{R}$, we first use
sequential packing to choose landmarks, in which
the points are scanned in order and each one is retained as a landmark only when it lies at least a prescribed separation distance $\epsilon>0$
from the landmarks already selected.  
A typical value of $\epsilon$ might be $.5$, noting
that the MDS results in a normalized scale which is in
units of standard deviation in the case of
Gaussian kernels.
 
We then generate an alpha complex using the dual active-set quadratic-programming method of \cite{carlsson2023alpha},
which allows for arbitrarily high dimension, as the input
depends only on respective dot products.
We have found it to be effective to only compute
up to the one-simplices, and fill out remaining ones
using the lazy construction, in which every simplex
is included provided all of its one-dimensional faces
exist in the alpha complex. The result may be used
to compute homology, to compute persistent homology
by using the lazy filtration associated to the vertex weights $b_i$, 
or to create an 
Isomap embedding based on a fixed super-level
set of its one-skeleton
\cite{Tenenbaum2000Isomap}.

\subsection{Coarse graining}

\label{sec:coarse_graining}

For very large Markov chains, we will use an additional
reduction by first replacing the transition matrix by
a smaller coarse-grained chain as follows. 
Suppose the original chain has
$N$ states and that we choose $n\ll N$ representative states. Let
$V$ be the $N\times n$ indicator matrix whose entry $v_{ij}$ is one
when the original state $i$ lies in the Voronoi cell of representative
state $j$, and zero otherwise. We form
\begin{equation}
\label{eq:mds}
    \bar{P}=[V^tPV], \ \ \bar{Q}=\exp(-t(\Id_n-\bar{P})),
    \ \ R=V\bar{R},\ \
    Q=[(\exp(-r_{ij}))]
\end{equation}
where as for vectors, 
the brackets mean that the rows are normalized to sum to one.
The second matrix $Q$ is then $N\times n$, so that we retain one row for every original state while working in the smaller probability simplex
$\Delta_n^+$.
Combining this with a low-rank factorization of $R$
makes it possible to work with large Markov-chain inputs
while storing only the logarithmic low-rank factors.

\section{Experiments}

\label{sec:experiments}

We demonstrate the method on two synthetic point clouds, a
one-dimensional family of images, and image data with nontrivial rotational
symmetry.  Each example follows the same pipeline: construct a local Markov
chain, pass to its positive continuous-time transition matrix using the matrix exponential, and sample 
from the transport shape. We then visualize or triangulate the results using multidimensional scaling applied to the logarithmic coordinates. 
For point clouds containing a few thousand points, computing
the matrix exponential and the sampled shape took a few seconds in our
implementation.  
The code for these tests can be found at the first author's website 
\url{math.ucdavis.edu/~ecarlsson/AlphaKL.zip}.

\subsection{Synthetic data maze}

\begin{figure}
\centering
\adjustbox{max width=.32\textwidth,max height=4cm}{
    \includegraphics[
    width=\linewidth,
    trim=0cm .5cm 0cm 0cm,
    clip]{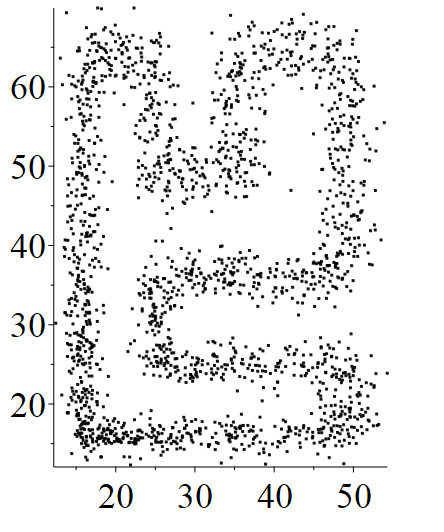}
}
\adjustbox{max width=.32\textwidth,max height=4cm}{
    \includegraphics[
    width=\linewidth,
    trim=0cm -1cm 0cm 0cm,
    clip]{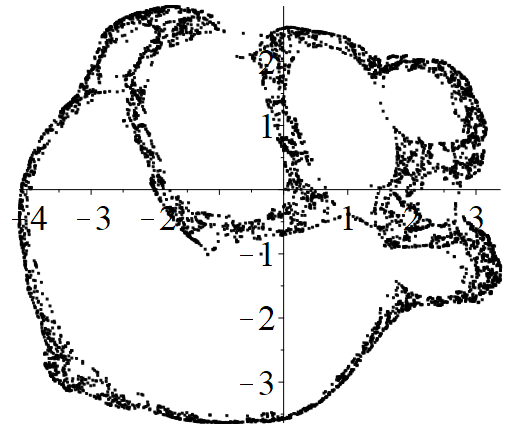}
}
\adjustbox{max width=.32\textwidth,max height=4cm}{
    \includegraphics[
    width=\linewidth,
    trim=0cm .5cm 0cm 0cm,
    clip]{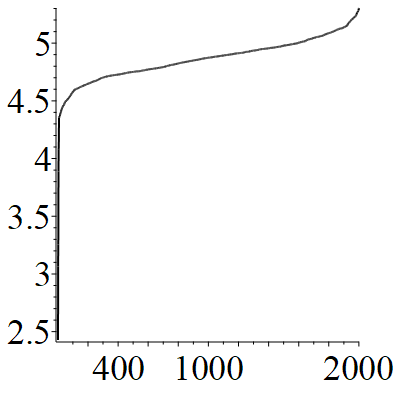}
}
\centering
\adjustbox{max width=.32\textwidth,max height=4cm}{
    \includegraphics[
    width=\linewidth,
    trim=0cm 2cm 0cm 0cm,
    clip]{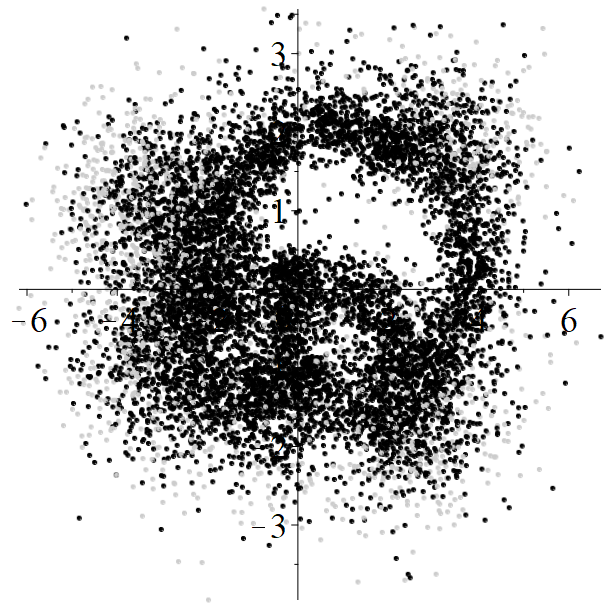}
}
\adjustbox{max width=.32\textwidth,max height=4cm}{
    \includegraphics[
    width=\linewidth,
    trim=0cm -1cm 0cm 0cm,
    clip]{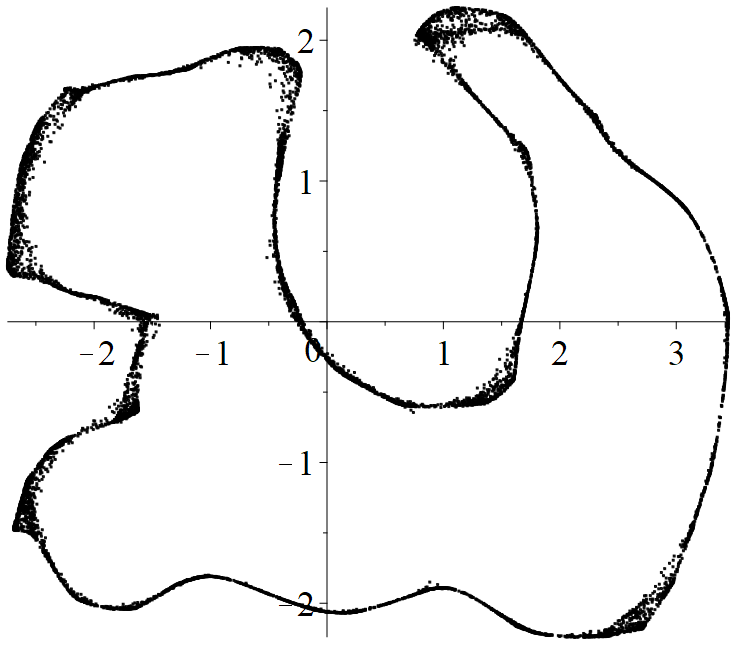}
}
\adjustbox{max width=.32\textwidth,max height=4cm}{
    \reflectbox{\includegraphics[
    width=\linewidth
    ]{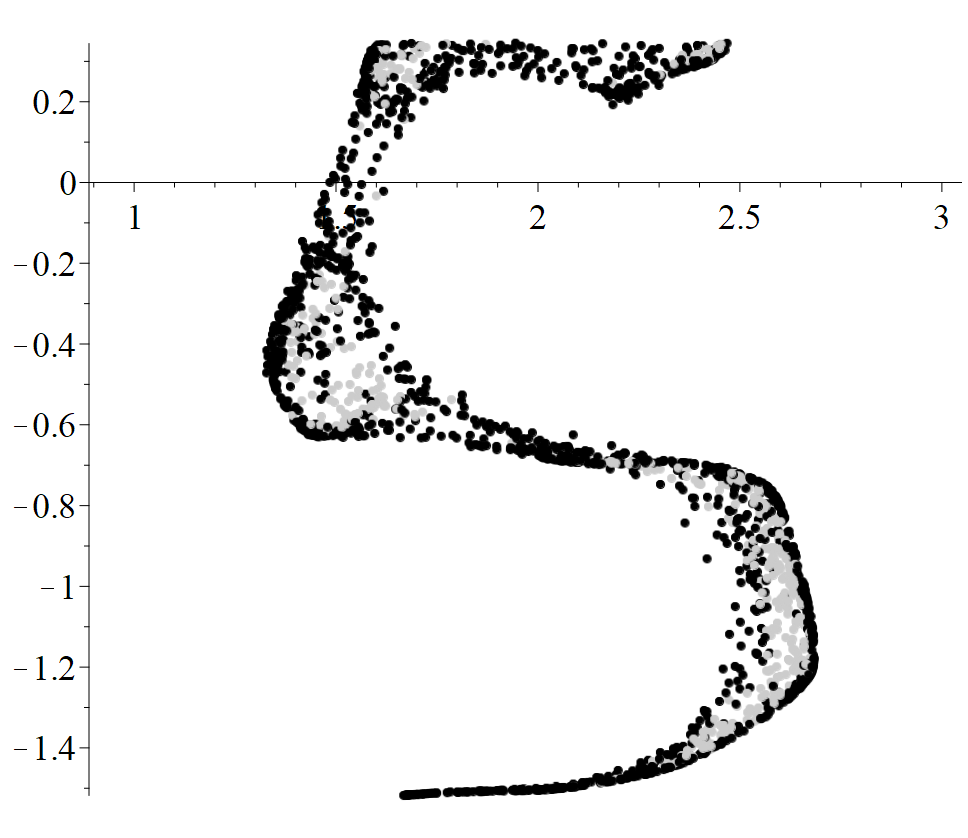}}
}
\caption{The data-maze experiment, from the input point cloud through
transported sampling and the final packed alpha complex.}
\label{fig:maze}
\end{figure}

Our first synthetic example is a noisy ``data maze'' consisting of $2000$ 
points in $\mathbb{R}^2$ arranged around a looping path, with nonuniform local 
scales. We formed the stochastic neighbors 
transition matrix with entropy $H=2$, 
computed $Q=P(t)$ at
$t=30$, and then reduced the rank to $r=100$ 
in the logarithmic factorization. 
Figure \ref{fig:maze} shows the successive stages.  
The upper row contains the original point cloud, 
the MDS embedding of the rows of $Q$ in logarithmic
coordinates, and a plot of the sorted values of $\psi(q_i)$
for each row $q_i$ of $Q$.

The lower left image shows
samples from the 
Gaussian distribution described in Method B of Section \ref{sec:sampling} in the same MDS coordinates, while the lower middle shows the considerably condensed
images of these points under the transport map $T$. We may select only points from either
$\mathcal{Q}(a)$ or $T(\mathcal{P}(a))$
 by filtering these points according to 
 $\zeta(T(\tilde{p}_i))$ and $\psi(\tilde{p}_i)$
 respectively.

The lower right frame illustrates that the super-level
set $\mathcal{P}(a)$ is indeed expected to be topologically the same as a circle for the value $a=4.0$. To generate the picture, we selected all
those sampled points $p=\tilde{p}_i$ for which 
$\psi(p)\geq a$, which are shown in gray.
For the remaining points shown in black, 
we moved along the
flow $p_s=[p^{1-s}q^{s}]$ for $q=T(p)$ as in 
Lemma \ref{lem:flow}, until obtaining the unique point at which $\psi(p_s)=4.0$. We discarded the answer
if there is no such point, in other words if
$\psi(q)<a$. We see that these points appear as the boundary of a tube surrounding the core shape,
which appears to be topologically equivalent to a circle.

\subsection{Uniform samples from the unit square}

The output of the previous example resembles UMAP in the sense that both methods can collapse the
thick, noisy maze to a point cloud arranged around its central circle. The next example illustrates a 
difference in terms of preserving the underlying geometry.

\begin{figure}[t]
\centering
    \begin{subfigure}[b]{.32\linewidth}
\includegraphics[scale=.21]{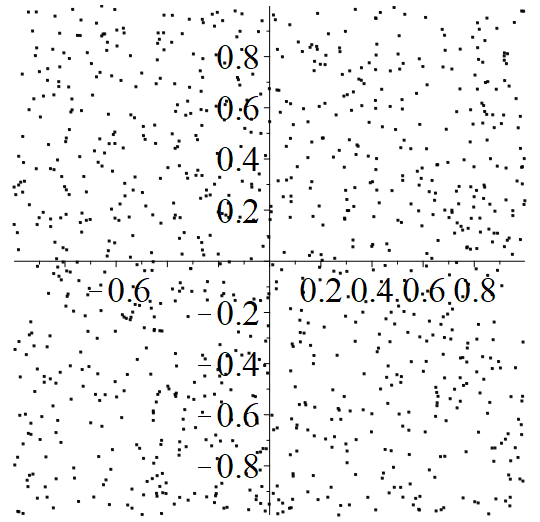}
\end{subfigure}
    \begin{subfigure}[b]{.32\linewidth}
        \includegraphics[scale=.305]{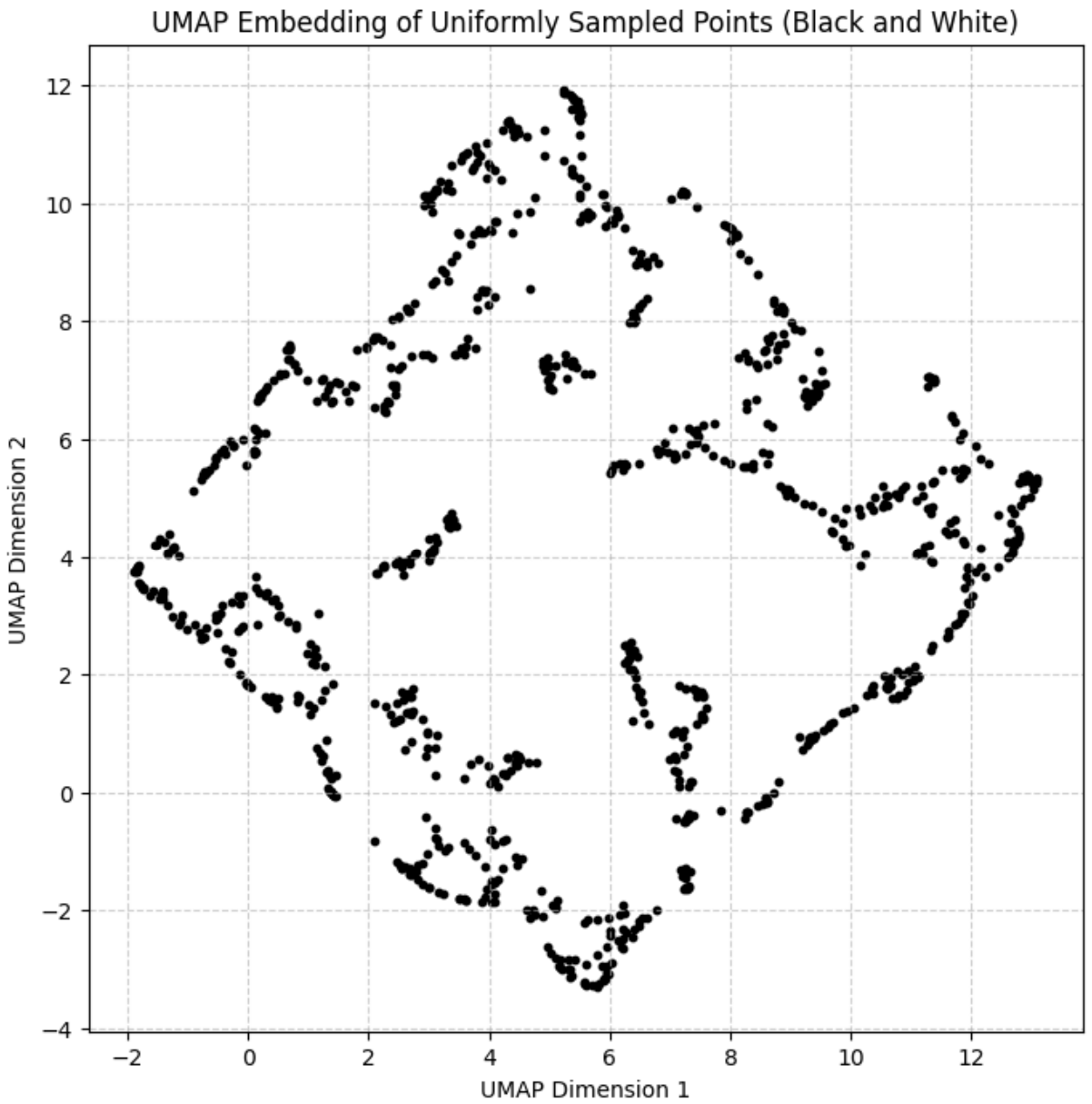}
    \end{subfigure}    
    \begin{subfigure}[b]{.32\linewidth}
\includegraphics[scale=.12]{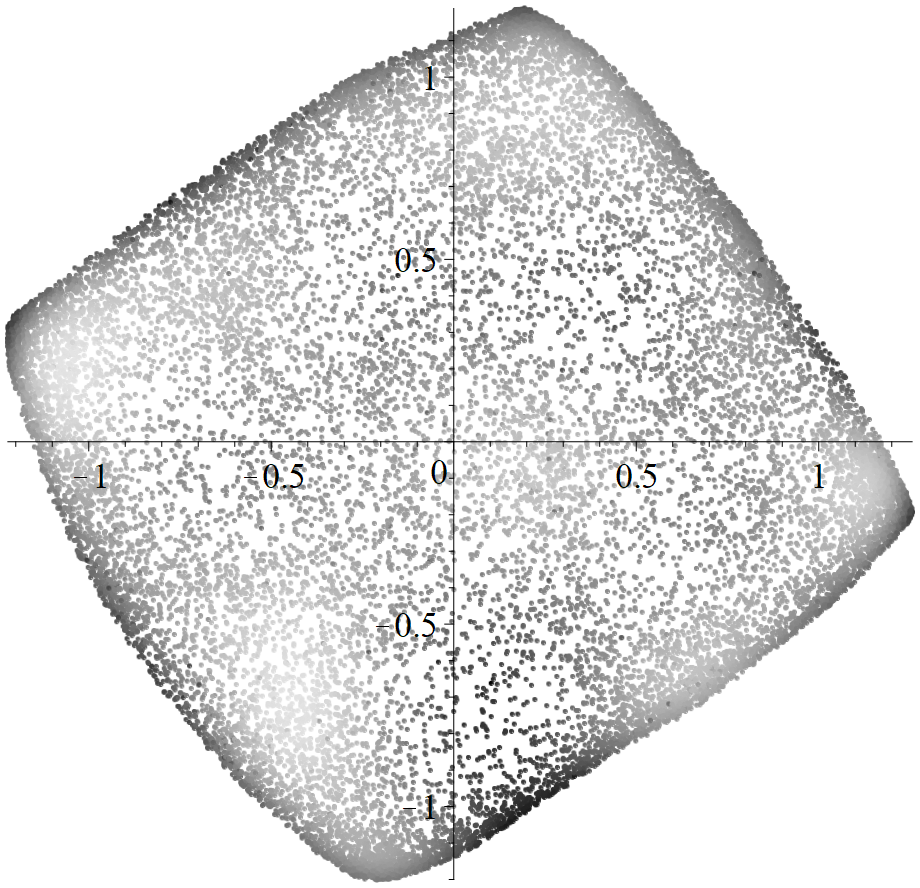}  
\end{subfigure} 

    \caption{1000 uniform samples from the unit square (left), its default UMAP embedding (middle),
    and 20000 samples from the transported shape (right).}
    \label{fig:square}
\end{figure}

For this, we sampled $1000$ points uniformly from
$[-1,1]^2$, shown in the left panel of Figure \ref{fig:square}.  The
middle panel applies UMAP with its standard choices
\texttt{n\_neighbors=15} and \texttt{min\_dist=.1},
showing that the embedding pinches
together regions that were separated in the square.
We then considered the transport shape
for the value of $H=2.0$ in the stochastic neighbors
walk, and $t=10$.
The right panel shows the result of sampling
20000 points from the shape using Method B again,
embedded using the first two MDS coordinates.
Adding a third coordinate shows some subtle warping,
but the shape is visibly two-dimensional.

We see that the samples in the transport shape 
are more uniform,
showing less pinching than appears in the second frame. Another feature is that the gaps may be filled
in by sampling more points from the shape than the
size of the original point cloud, in this case
20x more. Even if the goal is to obtain
fewer landmark points for the purposes of building a mesh, it is still an advantage to subsample from the larger collection of samples as in the right frame.

\subsection{Warmup: one-dimensional spaces of images}

We next considered an image dataset 
in which one object is observed from many viewing
angles.  For generic pictures, a full rotation should trace a topological circle, which is a standard test
for topological inference by persistent homology. It is also a setting which requires non-Euclidean kernel functions
which capture the changing metrics.

\begin{figure}[t]
\centering
\begin{subfigure}[b]{.32\linewidth}
\includegraphics[scale=.86]{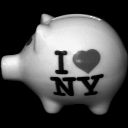}
\end{subfigure}
\begin{subfigure}[b]{.32\linewidth}
{\includegraphics[scale=.115,trim=0.0cm 2.0cm 0.0cm 4.0cm,clip]{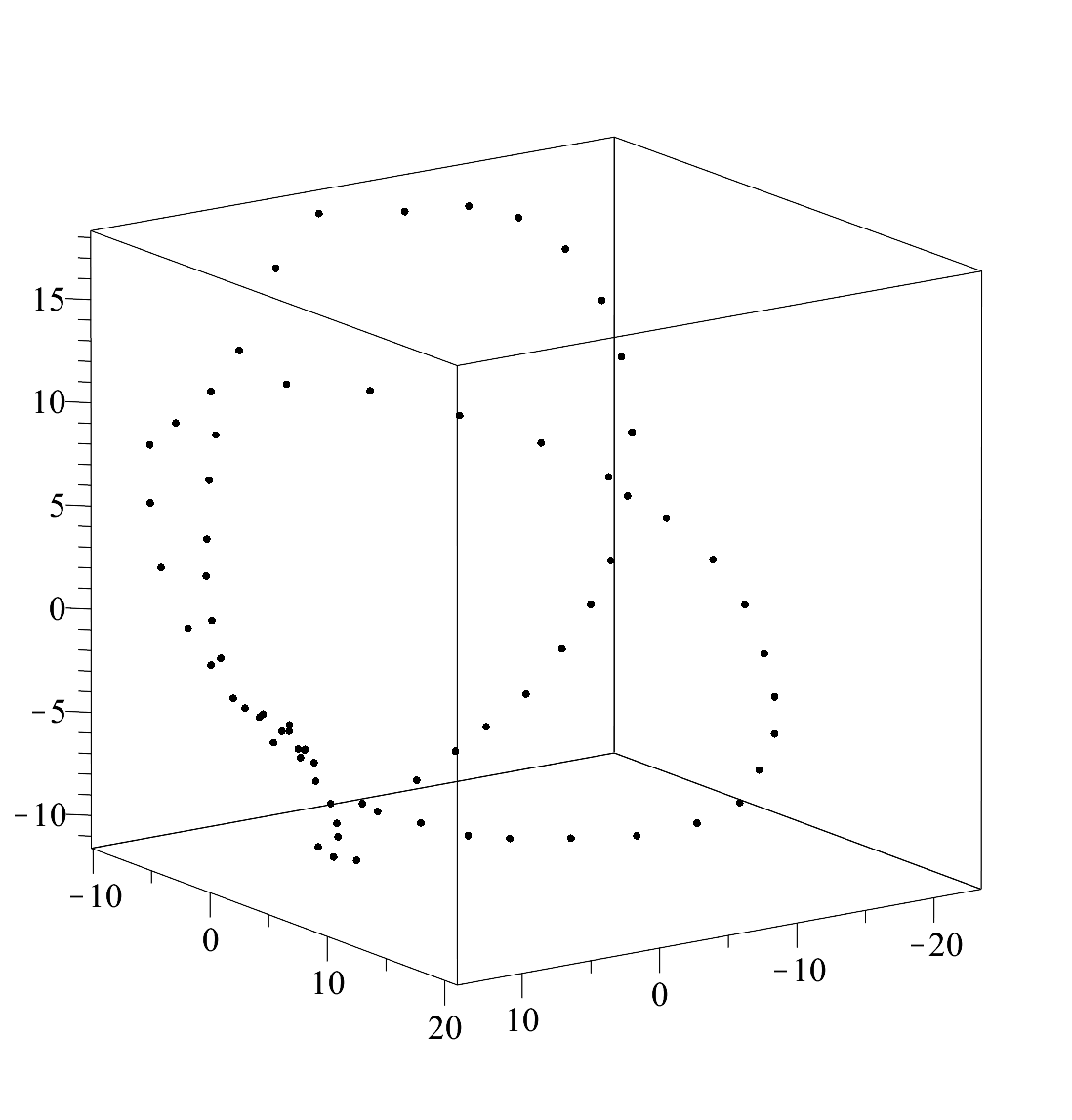}}
\end{subfigure}
\begin{subfigure}[b]{.32\linewidth}
{\includegraphics[scale=.10,trim=5.0cm 8.0cm 0.0cm 0.0cm,clip]{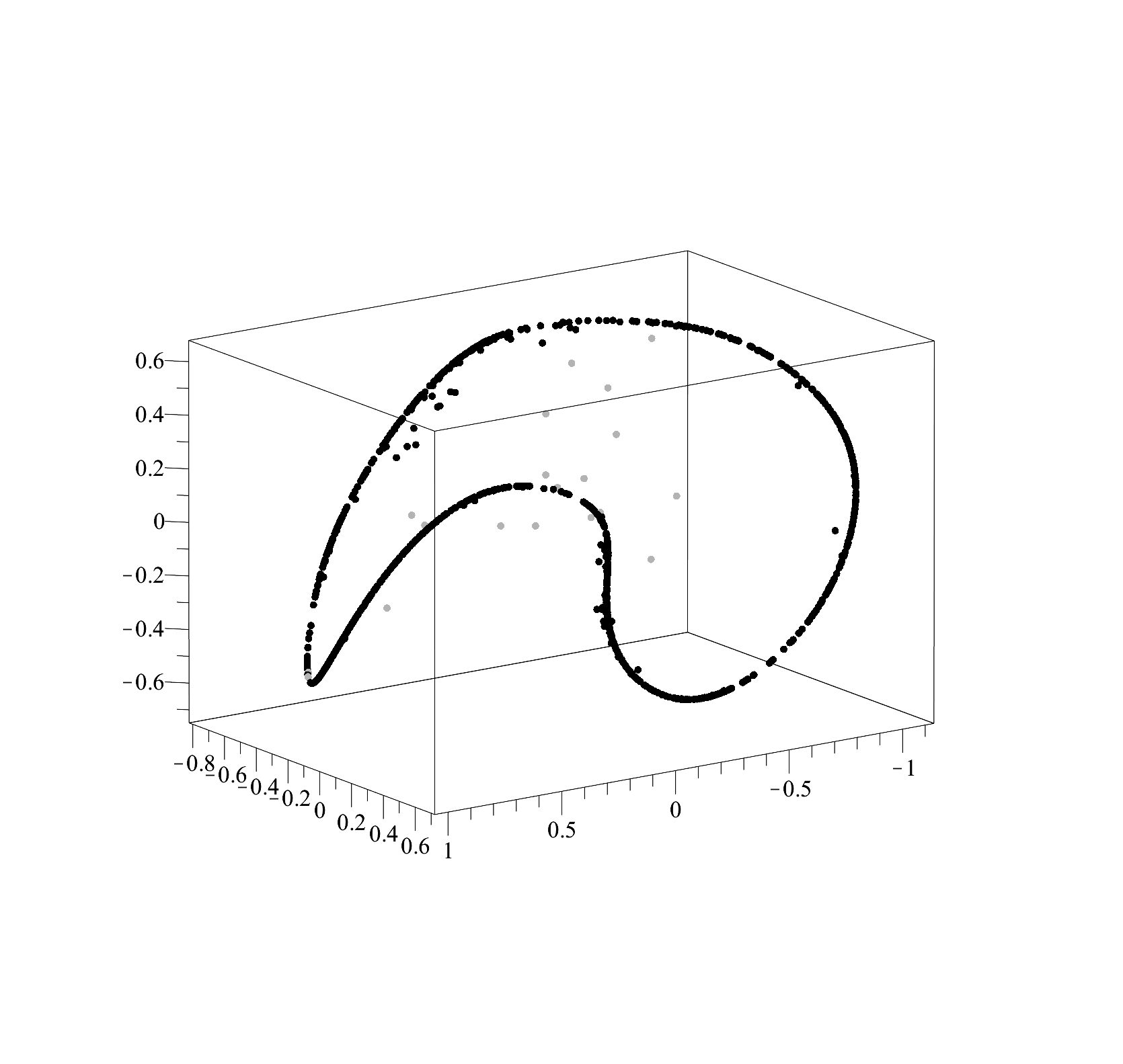}}
\end{subfigure}
\centering
\begin{subfigure}[b]{.32\linewidth}
\includegraphics[scale=.86]{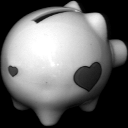}
\end{subfigure}
\begin{subfigure}[b]{.32\linewidth}
{\includegraphics[scale=.85]{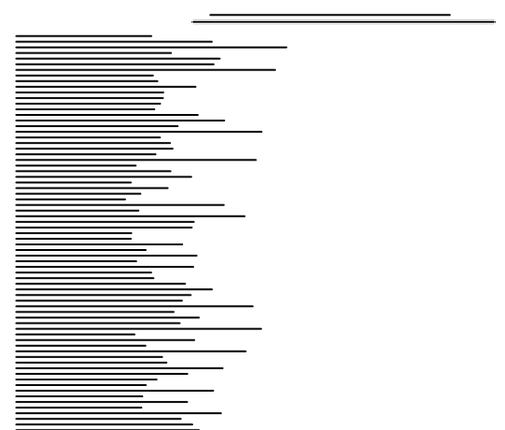}}
\end{subfigure}
\begin{subfigure}[b]{.32\linewidth}
{\includegraphics[scale=.28,trim=0.0cm 0.0cm 0.0cm 0.0cm,clip]{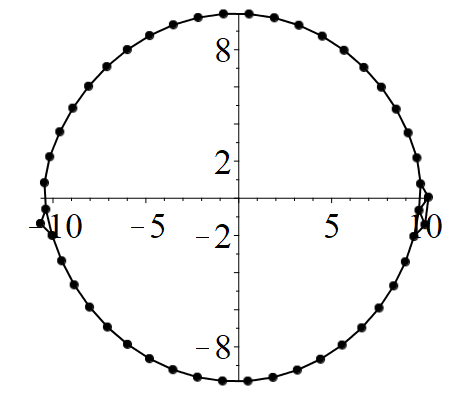}}
\end{subfigure}

    \caption{COIL images, their PCA and persistence analysis, and the sampled
    transport shape before and after Isomap.}
    \label{fig:coil}
\end{figure}

We used the COIL data set \cite{nene1996coil100}, in which an object is
photographed at $72$ angles over a complete $360^\circ$ rotation,
shown in the two
left panels of Figure \ref{fig:coil}. In the experiment in Oudot's INF556 notes \cite{oudotINF556TD5}, shown in the middle panels, the images were first
projected to three dimensions using PCA, and persistent homology
was then computed.
Two viewing angles turn out to be visually similar 
and collapse together under the projection. The resulting self-intersection creates two
one-dimensional persistent homology classes, which
are represented by the two longer bars at the top of the upper-middle frame.

To apply the transport method, we 
used stochastic neighbors with entropy $H=1.0$ and
the same thresholding procedure as in the maze example. 
The upper-right
panel shows the sampled transport shape, with white circles indicating samples whose filtration value was below a selected level. 
We then used
sequential packing and an alpha complex in the higher-dimensional MDS coordinates, followed by Isomap, to produce the single circle in the lower-right panel. Because the input contains only $72$ rows, the Gaussian/MDS resampling
from Method B was used to fill in the circle.  

\subsection{Tetrahedra and rotational symmetry}
\label{sec:tetrahedra}

\begin{figure}[t]
    \centering
\begin{subfigure}[b]{.32\linewidth}
\includegraphics[scale=.65]{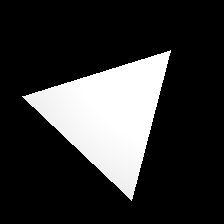}
\end{subfigure}
   \begin{subfigure}[b]{.32\linewidth}
\includegraphics[scale=.65]{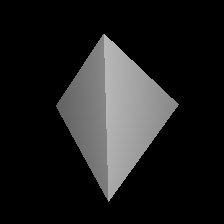}  
\end{subfigure}
    \begin{subfigure}[b]{.32\linewidth}
\includegraphics[scale=.65]{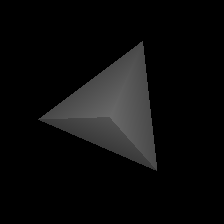}
\end{subfigure}           
    \centering

\begin{subfigure}[b]{.32\linewidth}
{\includegraphics[scale=.09,trim=0cm 4.5cm .8cm -1.1cm,clip]{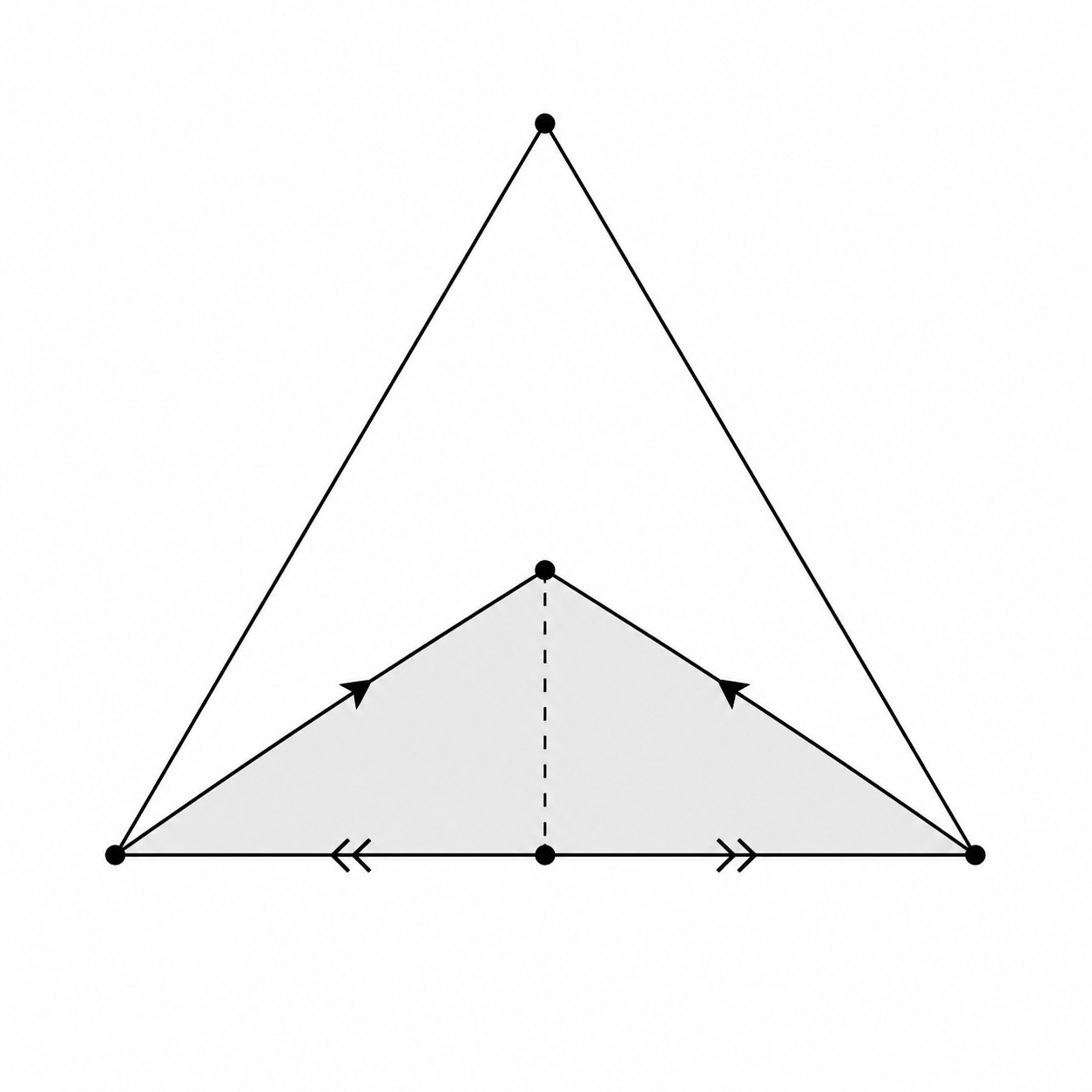}  }
\end{subfigure}
 \begin{subfigure}[b]{.32\linewidth}
    {\includegraphics[scale=.21,trim=7.5cm 6cm 7.7cm 7cm,clip]{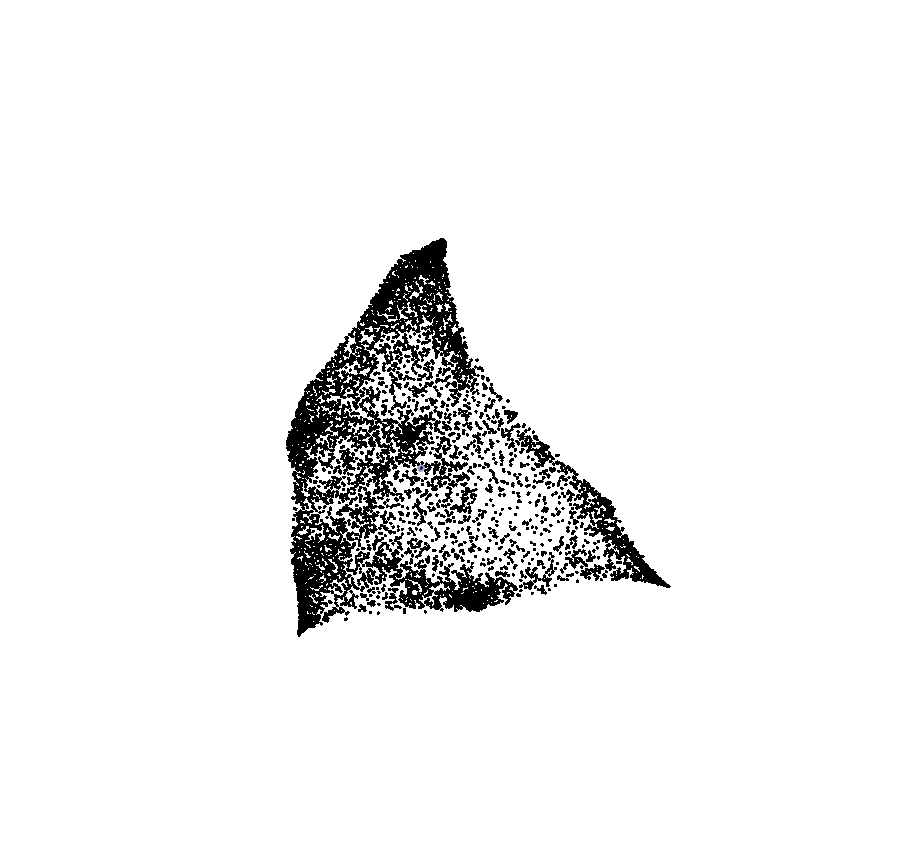}}
    \end{subfigure}   
    \begin{subfigure}[b]{.32\linewidth}
        {\includegraphics[scale=.241,trim=9.5cm 7cm 7.8cm 8cm,clip]{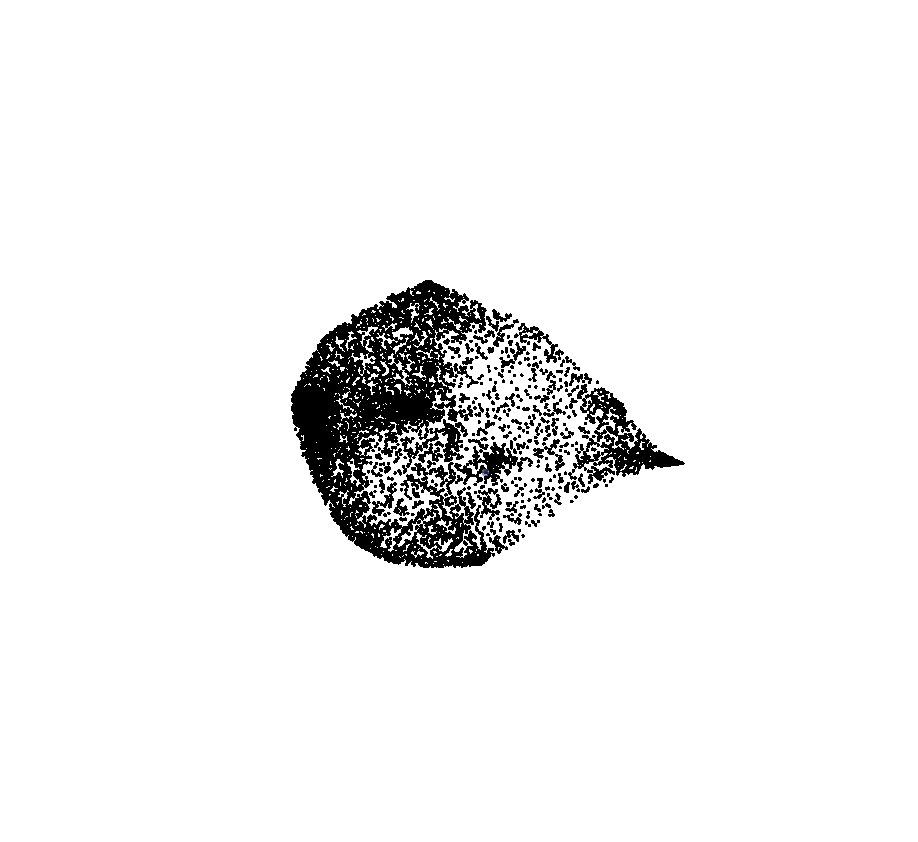}}
    \end{subfigure}    

    \caption{Three images from the tetrahedral SYMSOL dataset, a description of the fundamental domain, 
    and two views of the sampled quotient shape.}
    \label{fig:tetrahedron_distances}
\end{figure}

We next applied the construction to the SYMSOL symmetry data set \cite{wiersma2021symsol}, 
specifically the tetrahedron class, which consists of views of a tetrahedron from many orientations. 
Like the previous example, this is another
instance where standard $L^2$ distance falls off rapidly, making it a 
good use case for the moving kernel generated by the matrix exponential.
We applied the transport shape
to recover the geometry of these spaces
using both visual embeddings and a homology
calculation, using the methods described in Section \ref{sec:alphacomplex}.


We first generated a visual model of the transport shape, using a distance that 
is invariant under rotations of the image plane. To define this distance, we expanded each image in Gaussian-weighted Hermite modes through degree $k=20$, so that the resulting complex coordinates are compatible with the $S^1$ action obtained by rotating the image about the center. We then defined the distance between two images by minimizing their $L^2$ distance over this action, evaluating the minimum over 500 values of $\theta\in[0,2\pi]$. 
If there were no symmetries,
a point in the resulting shape would
be expected to capture the space of camera positions, which is a sphere, rather than the full $SO(3)$ of orientations.

We then applied stochastic neighbors to 2000
images using the rotation-invariant (squared) distance,
with the values of $H=2.0$, $t=30$.
We sampled 10000 points using method B,
and then projected the result to low dimensions using the MDS.
The tetrahedral rotational symmetries introduced three singular points, corresponding to views centered on a face, an edge, and a vertex, as in the top row of the Figure \ref{fig:tetrahedron_distances}. Geometrically, one can think of the shape as a fundamental spherical triangle with angles $30^\circ,30^\circ,120^\circ$, with appropriate edges identified. Its underlying topology is therefore that of a sphere, but with three cuspidal points, 
corresponding to the three types of special views. 
These features are clearly visible in the second
row.

We then removed the rotational invariance and instead used the transport construction as a source of samples for a topological calculation. 
For this we used all 45000 images and took the original 
points as the inputs to the transport map (method A), using also the coarse graining procedure from Section \ref{sec:coarse_graining} to reduce the number of states in the Markov chain to 10000, and took $H=3$ and $t=30$.
From the resulting data we constructed an alpha complex and computed its homology. The first three Betti numbers over 
$\mathbb Z/3$ were $(1,1,1)$ through dimension two, while over $\mathbb Z/2$ they were $(1,0,0)$, which is consistent with the expected topology of $SO(3)/A_4$, where $A_4$ is the rotational symmetry group of the tetrahedron. 
Using Javaplex in MATLAB, one can reconstruct representatives of these one-cycles
over $\mathbb{Z}/3$ from the vertices of the alpha complex and animate the corresponding sequence of tetrahedral views. Along such a cycle, one corner maps to itself, while the other three rotate in a cycle.

\bibliographystyle{plain}

\bibliography{refs}

\end{document}